\documentclass {article}
\usepackage{authblk}
\usepackage[utf8]{inputenc}
\usepackage[english]{babel}
\usepackage{amsmath, amsthm, amssymb}
\usepackage{titlesec}
\usepackage{color}
\usepackage[color,matrix,arrow]{xy}
\usepackage{amsgen}
\usepackage{amstext}
\usepackage{amsfonts}
\usepackage{graphicx}
\usepackage{xfrac}
\usepackage{float}
\usepackage{todonotes}

\usepackage{pgf}
\usepackage{tikz-cd}
\usepackage{tikz}
\usetikzlibrary{automata, calc, shapes, decorations.pathreplacing, decorations.pathmorphing, arrows, arrows.meta, positioning}
\usepackage{eqnarray}
\usepackage{array}
\usepackage{faktor}
\usepackage{ mathrsfs }

\usepackage{enumitem}
\usepackage{hyperref}
\usepackage[hypcap=false]{caption}
\xyoption{all}

\newcommand{\footrecall}[1]{
}

\titleformat*{\section}{\large\bfseries}
\titleformat*{\subsection}{\normalsize \bfseries}

\newcommand{\N}{\mathbb{N}}
\newcommand{\Z}{\mathbb{Z}}

\newcommand{\HH}{\mathcal{H}}
\newcommand{\DD}{\mathcal{D}}
\newcommand{\LL}{\mathcal{L}}
\newcommand{\RR}{\mathcal{R}}
\newcommand{\WP}{\text{WP}}

\newcommand{\DGeo}{\mathcal{D}\text{-}\Geo}
\newcommand{\RGeo}{\mathcal{R}\text{-}\Geo}
\newcommand{\LGeo}{\mathcal{L}\text{-}\Geo}
\newcommand{\HGeo}{\mathcal{H}\text{-}\Geo}
\newcommand{\JGeo}{\mathcal{J}\text{-}\Geo}
\newcommand{\SL}{\text{SL}}

\DeclareMathOperator{\Sub}{Sub}
\newcommand{\abs}[1]{\left|#1\right|}

\newcommand{\R}{\mathbb{R}}

\newcommand{\Geo}{\text{Geo}}
\newcommand{\FIM}{\text{FIM}}

\theoremstyle{definition}
\newtheorem{theorem}{Theorem}[section]

\newtheorem{definition}[theorem]{Definition}

\newtheorem{proposition}[theorem]{Proposition}

\newtheorem{lemma}[theorem]{Lemma}

\newtheorem{remark}[theorem]{Remark}

\title{Green languages and growth of free inverse monoids}

\author[1]{Corentin Bodart}
\affil[1]{Département de Mathématiques, Université du Luxembourg, Luxembourg
	
	\texttt{cobodart123@gmail.com}\\}

\author[2]{André Carvalho}
\affil[2]{Centro de Investigação em Matemática e Aplicações (CIMA)
	
	Departamento de Matemática, Escola de Ciências e Tecnologia da Universidade de Évora
	
	Rua Romão Ramalho, 59, 7000–671 Évora, Portugal
	
	\texttt{andre.carvalho@uevora.pt}\\}

\author[3]{Ana-Catarina C. Monteiro}
\affil[3]{Center for Mathematics and Applications (NOVA Math), NOVA School of Science and Technology (NOVA FCT)\\
	\texttt{acatarinacm@gmail.com}}

\date{}

\begin{document}

\maketitle

\begin{abstract}
Motivated by recent work on conjugacy languages in groups, we introduce a general framework for studying \emph{languages of representatives} associated to equivalence relations on finitely generated semigroups. After developing this notion in full generality, we particularize it to Green's relations and obtain the corresponding Green languages, which form the central objects of this article. For free inverse monoids, we describe these languages via Munn trees and establish several structural relations among them, including characterizations of the $\mathcal{D}$-, $\mathcal{R}$-, $\mathcal{L}$-, and $\mathcal{H}$-languages.

We then investigate the formal language-theoretic complexity of these Green languages. For free inverse monoids of rank at least two, we show that $\HGeo$ is context-free and co-context-free but not regular, while $\RGeo$, $\LGeo$, and $\DGeo$ are neither context-free nor co-context-free. In contrast, in the monogenic case, $\DGeo$ is regular and the remaining Green languages are deterministic context-free.

We further relate these results to the ShortLex language and the growth of free inverse monoids, proving in particular that $(\FIM_X,X)$ is growth tight. The article concludes with a collection of open problems and directions for future research.
\end{abstract}

 \section{Introduction}
 
The interaction between formal language theory and algebraic structures such as groups and monoids has been widely studied. A classical example is the word problem for virtually free groups, which is a context-free language \cite{[MS83]}. Geodesic languages have also played an important role in geometric group theory. In particular, it is showed that, for several classes of groups, such as hyperbolic groups, the set of geodesic representatives with respect to a finite generating set forms a regular language (see for example \cite{[ECHLPT92], [Can84]}).

In the setting of free inverse monoids, the word problem for the free inverse monoid has been investigated from the viewpoint of formal languages, providing a natural extension of these ideas beyond groups. In \cite{[B18]} it is proved that the word problem in a free inverse monoid is not context-free and for rank greater than 1, it is not poly-context-free. Further, in \cite{brough2026free} it is proved that the free inverse monoid of every finite rank has co-context-free word problem.

The study of conjugacy languages, i.e., languages of representatives of conjugacy classes in groups, has been a prominent topic of research (see, e.g.\ \cite{[CHHR16], [CM25]}). Within various families of groups, these languages have been rigorously defined and characterized according to their language-theoretic complexity. Driven by this line of research, this work aims to extend this framework to the setting of semigroups and Green relations.

Hence, after introducing in Section~2 some preliminary concepts and results, in Section~3 we develop, for semigroups, the notion of a \emph{language of representatives} associated to an equivalence relation. More precisely, given an equivalence relation $\sim$ on a finitely generated semigroup $S$ with finite generating set $X$, we define the language of shortest words representing elements in each equivalence class by

\[
\sim\!\text{-}\Geo_X(S)
= \Bigl\{w\in X^* : \ell(w)=\min\bigl\{\ell(u): u\in X^*,\ (u\pi)\sim (w\pi)\bigr\}\Bigr\}.
\]

We then particularize this definition to the Green's relations $\mathcal{H},\mathcal{L},\mathcal{R},\mathcal{D},$ and $\mathcal{J}$ for an arbitrary finitely generated semigroup $S$, introducing the corresponding Green languages as follows:
\begin{equation*}
    \begin{aligned}
        \JGeo_X(S) &=\bigl\{w\in X^*: \ell(w)=\min\{\ell(u): u\in X^*,\ (u\pi)\mathcal{J}(w\pi)\}\bigr\},\\
        \DGeo_X(S)&=\bigl\{w\in X^*: \ell(w)=\min\{\ell(u): u\in X^*,\ (u\pi)\mathcal{D}(w\pi)\}\bigr\},\\
        \RGeo_X(S)&=\bigl\{w\in X^*: \ell(w)=\min\{\ell(u): u\in X^*,\ (u\pi)\mathcal{R}(w\pi)\}\bigr\},\\
        \LGeo_X(S)&=\bigl\{w\in X^*: \ell(w)=\min\{\ell(u): u\in X^*,\ (u\pi)\mathcal{L}(w\pi)\}\bigr\},\\
        \HGeo_X(S)&=\bigl\{w\in X^*: \ell(w)=\min\{\ell(u): u\in X^*,\ (u\pi)\mathcal{H}(w\pi)\}\bigr\}.
    \end{aligned}
\end{equation*}

These languages will constitute the central objects of this article.

In the setting of free inverse monoids, Section 4 develops a detailed description of Green languages via Munn trees, providing a visual framework for working with these languages.

In Section~5, we begin the study of the Green languages and prove that, for any finite generating set $X$, the following properties hold:
\begin{itemize}
    \item $\DGeo(\FIM_X) = \RGeo(\FIM_X)\cap \LGeo(\FIM_X)$,
    \item $\RGeo(\FIM_X)\cup \LGeo(\FIM_X) \subset \HGeo(\FIM_X)$,
    \item $\RGeo(\FIM_X)=\LGeo(\FIM_X)^{-1}$.
\end{itemize}

Sections~6 and~7 are dedicated to the investigation of whether these languages are regular or context-free. For any finite set $X$ with $|X|>1$, we show that $\HGeo(\FIM_X)$ is context-free and co-context-free (but not regular) (see Theorems~\ref{geo_CF} and~\ref{thm_co-cont}). Also, we justify that if we fix any total order on $\tilde X$ and consider only the shortlex representatives (defining the language $\HH\text{-}\SL(\FIM_X)$) we still have a context-free and co-context-free language (see Theorem \ref{Sl_Geo}). In contract, we prove that any language having at least one geodesic representative of each $\DD$-class (resp. $\RR$-class or $\LL$-class) is not context-free. Clearly, this prove in particular that $\RGeo(\FIM_X)$, $\LGeo(\FIM_X)$ and $\DGeo(\FIM_X)$ are not context-free (see Theorems ~\ref{D-CF} and \ref{RL_CF}). Similarly, we show that any language $L$ satisfying $\DD\text{-}\SL(\FIM_X)\subseteq L\subseteq \DGeo(\FIM_X)$ is not co-context-free (see Theorem \ref{DGeo_not_CCF}) and that the languages $\RGeo(\FIM_X)$ and $\LGeo(\FIM_X)$ are also not co-context-free (see Theorem \ref{RL_CCF}).

In Section~9 we examine the monogenic case, that is, the free inverse monoid on one generator, and prove that $\DGeo(\FIM_X)$ is a regular language, while $\LGeo(\FIM_X)$, $\RGeo(\FIM_X)$ and $\HGeo(\FIM_X)$ are deterministic context-free, but not regular (see Theorem~\ref{thm_monog}).

The table below summarizes the results obtained regarding the Green languages on free inverse monoids.

\begin{center}
\begin{tabular}{ | m{6.8em} | m{1.9cm}| m{2.4cm} | m{2.4cm} | m{2.4cm} | } 
  \hline
  & $\HGeo$ & $\RGeo$ & $\LGeo$ & $\DGeo$ \\ 
  \hline
  Monogenic  & DCF & DCF & DCF & Regular \\ 
  \hline
  Non Monogenic & CF + co-CF & $\neg$CF\ +\ $\neg$co-CF & $\neg$CF + $\neg$co-CF & $\neg$CF + $\neg$co-CF \\ 
  \hline
\end{tabular}
\end{center}

\vspace{0.3cm}
\medskip

In the last section, we study the growth of $\FIM_X$. In \cite{Growth_FIM}, Kambites, Nyberg-Brodda, Szakács and Webb studied the exponential growth rate
\[ \alpha(\FIM_X,X) = \limsup_{n\to\infty} \sqrt[n]{\gamma(n)} \]
where $\gamma(n) = \#\{g\in \FIM_X : \abs{g}_X=n\}$. They prove that $\alpha(\FIM_X,X)$ is algebraic, and obtain asymptotic results when $\abs X\to\infty$. More strongly, the fact that the language $\SL(\FIM_X)$ of ShortLex representatives is unambiguously context-free implies that the growth series 
\[ \Gamma_{\FIM_X}(z) = \sum_{n=0}^\infty \gamma(n)\cdot z^n \]
is algebraic, recovering a result of Lau \cite{Lau}. (This is similar to results of Parry for the wreath product $\Z/2\Z\wr F_X$ \cite{Parry}. This similarity is explained by the fact that $\FIM_X\le L\wr F_X$, where $L=\mathrm{Mon}\langle e\mid e^2=e\rangle$.) Using the context-free language $\SL(\FIM_X)$, we prove that $(\FIM_X,X)$ is \emph{growth tight}, answering Question 5.9 of \cite{Growth_FIM}. Explicitly, we prove that if $M$ is an inverse monoid generated by a finite set $X$, with $\abs X\ge 2$, then $\alpha(M,X) \le \alpha(\FIM_X,X)$ with equality if and only if $M=\FIM_X$ (Theorem \ref{growth}).

The same result is known for free groups \cite[Section 2]{grigorchuk1997problems} and free monoids, however these proofs use the Nielsen-Schreier theorem which is not available here. Instead, we study the singularities of the growth series $\Gamma_{\FIM_X}(z)$ and use tools from analytic combinatorics.

\bigskip

We believe that this article sets out a new framework that opens the way to a broad range of further developments, within which several new research directions can naturally emerge. The paper concludes in Section~11 with a list of open questions together with a brief discussion of ongoing work, that already builds upon the ideas introduced in this work.

\section{Preliminaries}
In this section we present basic definitions and results on formal languages, and free inverse monoids and Munn trees, that will be essential throughout the paper.
\subsection{Formal languages}
\subsubsection{Regular languages}
An alphabet $X$ is a set and its elements are called letters. A word over $X$ is simply a finite sequence of letters of $X$.
Given an alphabet $X$, the Kleene star ($*$) on $X$ generates the set of all words over $X$, i.e.,
$$
X^* = \bigcup_{n \ge 0} X^n,
$$
where $X^0 = \{\varepsilon\}$ and $\varepsilon$ denotes the empty sequence. Together with concatenation, $X^*$ forms the free monoid generated by $X$.

A language $L \subseteq X^*$ is \textit{regular} if it can be constructed from $\emptyset$, $\{\varepsilon\}$, and $\{a\}$ (for all $a \in X$) through a finite sequence of regular operations: union ($\cup$), concatenation ($\cdot$), and Kleene star ($*$). Recall that regular languages are closed under finite unions, intersections, and complementation. Furthermore, the class of regular languages coincides precisely with the languages accepted by deterministic finite automata (DFAs) (for more details, the reader is referred to \cite{[HU79]}).

	Also, every regular language $L$ satisfies the \textit{Pumping Lemma}, which can be stated as follows.
    \begin{lemma}[Pumping Lemma for Regular Languages]
Let $L$ be a regular language. Then there exists a constant $p \geq 1$ such that every word $w \in L$ with $|w| \geq p$ can be written in the form
\[
w = xyz
\]
satisfying the following conditions:
\begin{enumerate}
    \item $|xy| \leq p$;
    \item $|y| \geq 1$;
    \item For all $i \geq 0$, the word $x y^i z$ belongs to $L$.
\end{enumerate}
\end{lemma}
\subsubsection{Context-free languages}
Recall that a \emph{context-free language} is a language generated by a \emph{context-free grammar}, which is a quadruple $G = (V, \Sigma, R, S)$ where $V$ is a finite set of variables, $\Sigma$ is a finite set of terminals, $R$ is a finite set of production rules of the form $A \to \alpha$ with $A \in V$ and $\alpha \in (V \cup \Sigma)^*$, and $S \in V$ is the start symbol. A string $w \in \Sigma^*$ is said to be generated by $G$ if $S \to^* w$, and the \emph{language of the grammar} $G$ is defined as

$$
L(G) = \bigl\{ w \in \Sigma^* \mid S \to^* w \bigr\},
$$

where $\to^*$ denotes the reflexive and transitive closure of the derivation relation $\to$, that is, zero or more derivation steps. In a similar way, we define $\to^+$ as the transitive closure of $\to$, that is, one or more derivation steps. Let us list a few more properties of grammars:
\begin{itemize}[leftmargin=5mm]
\item A grammar is \emph{reduced} if, for all $A\in V$, there exists $\alpha,\beta\in (V\cup\Sigma)^*$ and $w\in\Sigma^*$ such that $S\to^* \alpha A\beta$ and $A\to^* w$. For every grammar $G$, there exists a reduced grammar $\tilde G$ such that $L(G)=L(\tilde G)$. We will always assume that our grammars are reduced.

\item A grammar is \emph{unambiguous} if, for every word $w\in L(G)$, there exists a unique derivation
\[
S \to \alpha_1 \to \alpha_2 \to \ldots \to \alpha_\ell \to w
\]
such that, at every step, the production rule is used on the leftmost variable of $\alpha_i$.

\item A grammar is \emph{non-linear} if there exists $A\in V$ such that $A\to^* \alpha A\beta A\gamma\in (V\cup\Sigma)^*$.

\item To each context-free grammar, we associate a \emph{dependency (di)graph} $D(G)$ with
	\begin{itemize}[leftmargin=5mm]
		\item vertex set $V$,
		\item if $A\to \alpha B\beta\in (V\cup\Sigma)^*$, then we put an oriented edge $A\leadsto B$.
	\end{itemize}
A grammar is \emph{ergodic} if $D(G)$ is strongly connected.

\item A grammar is \emph{irreducible} if it is both ergodic and non-linear.
```

\end{itemize}

A language $L$ is \emph{context-free} if there exists a context-free grammar $G$ such that $L = L(G)$.

A \emph{pushdown automaton} is a tuple

$$
\mathcal{A}=(Q,\Sigma,\Gamma,\delta,q_0,Z_0,F),
$$

where $Q$ is a finite set of states, $\Sigma$ is the input alphabet, $\Gamma$ is the stack alphabet, $q_0\in Q$ is the initial state, $Z_0\in\Gamma$ is the initial stack symbol, $F\subseteq Q$ is the set of final states, and

$$
\delta:Q\times(\Sigma\cup\{\varepsilon\})\times\Gamma
\rightharpoonup Q\times\Gamma^*
$$

is a partial transition function. A \emph{one-counter automaton} is a pushdown automaton

$$
\mathcal A=(Q,\Sigma,\Gamma,\delta,q_0,Z_0,F)
$$

whose stack alphabet is of the form

$$
\Gamma=\{Z_0,X\},
$$

where $Z_0$ is a bottom-of-stack symbol and $X$ is the counter symbol. Thus, the stack has the form $X^nZ_0$ for some $n\in\mathbb N_0$, and the integer $n$ may be regarded as the value of a counter. Transitions may increase the counter by one, decrease it by one when it is positive, or leave it unchanged; the symbol $Z_0$ is never removed from the stack.

A one-counter automaton is said to be \emph{deterministic} if, from every configuration, there is at most one possible transition. Equivalently, its transition function
\[
\delta:Q\times(\Sigma\cup\{\varepsilon\})\times\Gamma
\rightharpoonup Q\times\Gamma^*
\]
is a partial function and, for every $q\in Q$ and $Y\in\Gamma$, if
$\delta(q,\varepsilon,Y)$ is defined, then $\delta(q,a,Y)$ is undefined
for every $a\in\Sigma$.

A language is called \emph{deterministic one-counter} if it is accepted by a deterministic one-counter automaton.

Every deterministic one-counter language is deterministic context-free, and therefore context-free. Moreover, since the class of deterministic context-free languages is closed under complementation, every deterministic one-counter language is also co-context-free.
Two of the most well-known properties of the class of context-free languages are that it is closed under intersection with regular languages and under finite unions. For further details on context-free languages see, for example, \cite{[Ber79]} or \cite{[HU79]}.

A language is said to be \emph{co-context-free} if its complement is context-free. As a consequence of the fact that context-free languages are closed under finite unions, it follows that co-context-free languages are closed under finite intersections.

Below we state two necessary conditions for a language to be context-free, which will be particularly useful in this paper for proving that given languages are not context-free.

\begin{lemma}[Pumping Lemma for Context-Free Languages]
Let $L$ be a context-free language. Then there exists a constant $p \geq 1$ such that every word $w \in L$ with $|w| \geq p$ can be written in the form
\[
w = uvxyz
\]
satisfying the following conditions:
\begin{enumerate}
    \item $|vxy| \leq p$;
    \item $|vy| \geq 1$;
    \item For all $i \geq 0$, the word $uv^i x y^i z$ belongs to $L$.
\end{enumerate}
\end{lemma}

\begin{lemma}[Ogden's Lemma]
Let $L\subseteq\tilde X^*$ be a context-free language. Then there exists a constant $p \geq 1$ such that for every word $w \in L$ with $|w| \geq p$, and for every choice of at least $p$ distinguished positions in $w$, there exists a decomposition
\[
w = uvxyz
\]
satisfying the following conditions:
\begin{enumerate}
    \item The substrings $v$ and $y$ together contain at least one distinguished position;
    \item The substring $vxy$ contains at most $p$ distinguished positions;
    \item For all $i \geq 0$, the word $uv^i x y^i z$ belongs to $L$.
\end{enumerate}
\end{lemma}

An important framework in formal language theory is the \emph{Chomsky hierarchy}, which classifies languages according to their expressive power and the types of grammars that generate them. It consists of four classes, regular, context-free, context-sensitive, and recursively enumerable languages, arranged in a hierarchy by inclusion, where each class properly contains the previous one. In particular, every regular language is context-free, and every context-free language is context-sensitive. We may also consider the class of co-context free languages, that is known not to coincide with any level of the Chomsky hierarchy. The diagram below illustrates the relationships between these classes.
\begin{center}
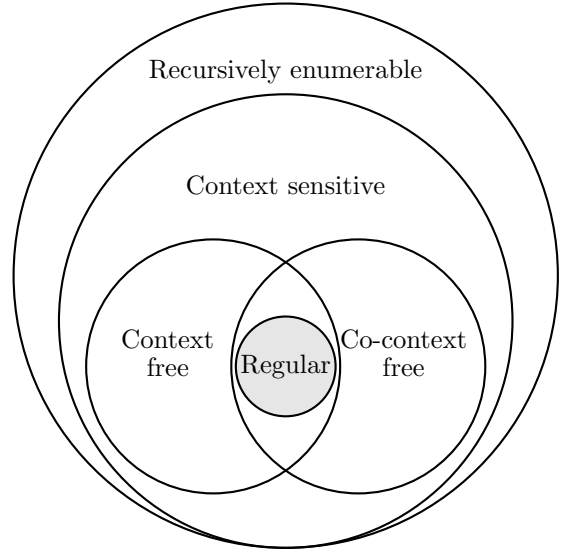

    \begin{tikzpicture}[scale=1.2]
        \draw[thick] (0,1) circle (3.0);
        \node at (0,3.25) {Recursively enumerable};
        \draw[thick] (0,0.5) circle (2.5);
        \node at (0,2) {Context sensitive};
        \draw[thick] (-0.8,0) circle (1.4);
        \node at (-1.3,0.3) {Context};
        \node at (-1.3,0) {free};

        \draw[thick] (0.8,0) circle (1.4);
        \node at (1.3,0.3) {Co-context};
        \node at (1.3,0) {free};

        \draw[thick, fill=gray!20] (0,0) circle (0.55);
        \node at (0,0) {Regular};
    \end{tikzpicture}
    \captionof{figure}{Relations between classes of languages}
\end{center}

\subsection{Free inverse monoids}
In this section, we recall the basic definitions and notation concerning inverse monoids, with particular emphasis on free inverse monoids.

Throughout the paper, $X$ will always denote a finite set.

An \emph{inverse monoid} is a monoid $S$ such that for every element $s \in S$ there exists a unique element $s^{-1} \in S$ satisfying
\[
s s^{-1} s = s \quad \text{and} \quad s^{-1} s s^{-1} = s^{-1}.
\]
The element $s^{-1}$ is called the \emph{inverse} of $s$. A \emph{free inverse monoid} on a set $X$ is an inverse monoid $\FIM_X$ together with a map $\iota : X \to \FIM_X$ satisfying the universal property that, for every inverse monoid $S$ and every map $\varphi : X \to S$, there exists a unique homomorphism $\overline{\varphi} : \FIM_X \to S$ such that $\overline{\varphi} \circ \iota = \varphi$.

Let $\tilde{X} = X \cup X^{-1}$, where $X^{-1} = \{x^{-1} \mid x \in X\}$ is a disjoint copy of $X$. We denote by $\tilde{X}^*$ the free monoid on $\tilde{X}$. Given an inverse monoid $S$ generated by $X$, we denote by
\[
\pi : \tilde{X}^* \to S
\]
the canonical surjective homomorphism extending the map $x \mapsto x$ and $x^{-1} \mapsto x^{-1}$ for all $x \in X$. We will write $\ell(u)$ to denote the length of a word $u\in \tilde X$.

As usual, $E(S)$ represents the set of idempotents of $S$, that is,
\[
E(S) = \{ e \in S \mid e^2 = e \}.
\]
In an inverse monoid, the set $E(S)$ forms a commutative subsemigroup.

We now recall the Green's relations on a semigroup $S$. For $a,b \in S$, we define:
\begin{align*}
a \,\mathcal{L}\, b &\iff S^1 a = S^1 b, \\
a \,\mathcal{R}\, b &\iff a S^1 = b S^1, \\
a \,\mathcal{J}\, b &\iff S^1 a S^1 =S^1 b S^1,
\end{align*}
where $S^1$ denotes $S$ if it is a monoid, and $S$ with an identity adjoined otherwise.

The relations $\mathcal{L}$, $\mathcal{R}$ and $\mathcal{J}$ are equivalence relations on $S$.

The relations $\mathcal{H}$ and $\mathcal{D}$ are defined by
\[
\mathcal{H} = \mathcal{L} \cap \mathcal{R}
\quad \text{and} \quad
\mathcal{D} = \mathcal{L} \vee \mathcal{R},
\]
where $\vee$ denotes the join of equivalence relations. Equivalently, $\mathcal{D}$ is the smallest equivalence relation containing both $\mathcal{L}$ and $\mathcal{R}$. Also, recall that in a free inverse monoid $\mathcal{D}=\mathcal{J}$ \cite{[M72]}.

\subsection{Munn trees}

Given a finite set $X$, the Munn tree of a word in $\tilde X^*$ representing an element on $\FIM_X$ is a birooted tree. It is constructed by first fixing an initial vertex, and then, for each successive letter of the word, adding a new edge labelled by that letter. An edge is added unless the letter is the inverse of the label of an edge adjacent to the current vertex; in that case, that edge is crossed and no new edge is added. Once all letters are read, the final vertex is also fixed. For more details see \cite{[M72]}.

For any $u\in X^*$, the Munn tree representing $u\pi$, where $\pi$ is the natural epimorphism between $\tilde X^*$ and $\FIM_X$ will be denoted by $MT(u)=(\Gamma(u),\alpha(u),\beta(u))$, where $\alpha(u)$ and $\beta(u)$ are the initial and final vertices, respectively. For example, if $u=aabb^{-1}a$, the triple $(\Gamma(u),\alpha(u),\beta(u))$ can be represented by:

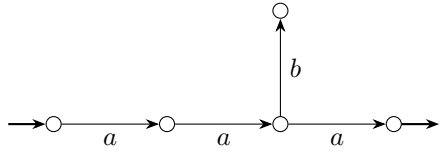
\begin{center}
	\begin{tikzpicture}[
		vertex/.style={circle, draw, inner sep=2pt},
		>={Stealth[length=5pt]}
		]

		\node[vertex] (e) at (0,0) {};
		\node[vertex] (a1) at (1.5,0) {};
		\node[vertex] (a2) at (3,0) {};
		\node[vertex] (a3) at (4.5,0) {};
		\node[vertex] (b)  at (3,1.5) {}; 
		\draw[->] (e) -- node[below] {$a$} (a1);
		\draw[->] (a1) -- node[below] {$a$} (a2);
		\draw[->] (a2) -- node[below] {$a$} (a3);
		\draw[->] (a2) -- node[right] {$b$} (b);
		
		\draw[->, thick] (-0.6,0) -- (e);
		\draw[->, thick] (a3) -- +(0.6,0);
		
	\end{tikzpicture}
\end{center}

In \cite{[M72]} the following characterization of the Green's relations regarding Munn tress is proved. 
\begin{theorem}{\cite{[M72]}}\label{Munn_Green}
	For any $u,v\in \tilde X^*$, then
	\begin{enumerate}
		\item $(u\pi)\mathcal{D} (v\pi)$ if and only $u\pi\mathcal{J}v\pi$ if and only if $\Gamma(u)=\Gamma(v)$;
		\item $(u\pi)\mathcal{R}(v\pi)$ if and only if $\Gamma(u)=\Gamma(v)$ and $\alpha(u)=\alpha(v)$;
		\item $(u\pi)\mathcal{L}(v\pi)$ if and only if $\Gamma(u)=\Gamma(v)$ and $\beta(u)=\beta(v)$;
		\item $(u\pi)\mathcal{H}(v\pi)$ if and only if $u\pi=v\pi$ if and only if $MT(u)=MT(v)$.
	\end{enumerate}
\end{theorem}

Note that in a Munn tree, returning to the same vertex generates an idempotent.

Throughout this paper, by a \emph{path} (or \emph{walk}) we mean any sequence of consecutive edges in the Munn tree, and by a \emph{simple path} we mean a path that does not repeat vertices (and hence does not repeat edges). Since we are working with trees, there is a unique simple path connecting any two vertices. As usual, we will say that a vertex is a \textit{leaf} if it has degree 1.

Given two vertices $p$ and $q$ in a Munn tree, we will denote by $d(p,q)$ the length of the simple path between them, where by length we mean the number of edges on the simple path.

It is well known that any geodesic word $u$ in a free inverse monoid can be decomposed in the form
\[
u = e_1 u_1 e_2 \cdots e_n u_n e_{n+1},
\]
where $u_1 u_2 \cdots u_n$ is a reduced word when projected to the free group, $e_i \in E(\FIM)\setminus\{1\}$ for $i \in \{2, \ldots, n\}$, and $e_1, e_{n+1} \in E(\FIM)$. This decomposition is unique (see~\cite{[PS05]}), a fact that will be  recalled later in this paper. In this paper we will use the notation $\bar u=u_1u_2\ldots u_n$.

Throughout this article, we will often view words as paths in their associated Munn tree.
Moreover, we will follow the notation of \cite{[PS05]} and, given two words \( u \) and \( v \), we write \( u = v \) when they are equal as words, and \( u \sim v \) when they represent the same element in the free inverse monoid.

\begin{lemma}\label{passar3x}
    If a  word over $\widetilde X^*$ crosses an edge of its Munn tree at least three times, then it is not a geodesic.
\end{lemma}
\begin{proof}
    If a word $w$ crosses three times an edge labeled by $a\in \tilde X$, then it can be written as
	\[
	w = w_1 a l_1 a^{-1} l_2 a w_2,
	\]
	where $l_1$ and $l_2$ are idempotents. Since idempotents commute, we have
	$$
	w= w_1\cdot al_1a^{-1}\cdot l_2\cdot aw_2 \sim w_1l_2 a\cdot l_1\cdot a^{-1}a\cdot w_2 \sim w_1 l_2\cdot aa^{-1}a\cdot l_1w_2 \sim w_1l_2al_1w_2,
	$$
	and so, $w$ is not a geodesic.	
\end{proof}

For future reference, we now formalise several concepts related to Munn trees.
\begin{definition}\label{def_main/idempotent}
  Let $u = e_1 u_1 e_2 u_2 \cdots e_n u_n e_{n+1}$ be a geodesic in a free inverse monoid, and consider its associated Munn tree, where $u_1 u_2 \cdots u_n$ is a reduced word under the natural projection to the free group, $e_i \in E(\FIM)\setminus \{1\}$ for $i \in \{2, \ldots, n\}$, and $e_1, e_{n+1} \in E(\FIM)$. We define
    \begin{enumerate}
        \item Main path - The unique simple path between the initial and final vertices.
        \item Idempotent edge - Any edge that does not belong to the main path
         \item Idempotent branch - Every idempotent edge belongs to an idempotent branch and only idempotent edges may belong to an idempotent branch. An idempotent branch may contain several idempotent branches. A new idempotent branch starts in the following situations:
\begin{enumerate}
    \item First, when an idempotent edge with label $a$ has one of its vertices, say $v$, on the main path; in this case, the branch consists of the entire subtree containing the vertices (and corresponding edges) whose simple path to $v$ crosses that edge labeled $a$. 
    Such an idempotent branch (that may contain other idempotent branches) is said to be a \textit{main idempotent branch}.
    \item Second, in a main idempotent branch, every vertex $v$ of degree $g>2$ gives rise to $g-1$ new idempotent branches,
    each starting at one of the edges adjacent to that vertex, except for the edge, say $a$, that belongs to the simple path between $v$ and the main path, i.e., between $v$ and the closest vertex of the main path.
    
    Each such branch starting at an edge of label $b$ with initial vertex $v$ contains all vertices (and corresponding edges) whose simple path to $v$ crosses the edge $b$.
\end{enumerate}

For example, the Munn tree presented below, whose main path is represented by the blue edges, has two main idempotent branches, shown in red and green. Each of these branches contains new idempotent branches. In particular, the vertices $v_1$ and $v_2$ each give rise to two additional branches, while the vertex $v_3$ gives rise to three additional branches.

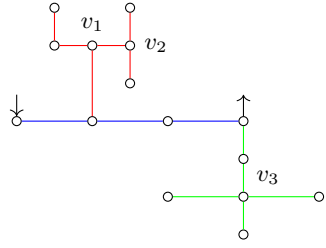
\begin{center}
\begin{tikzpicture}[
	vertex/.style={circle, draw, inner sep=1.2pt, fill=white},
	every node/.style={font=\small}
]

\node[vertex] (a) at (-2,0) {};
\node[vertex] (b) at (-1,0) {};
\node[vertex] (c) at (0,0) {};
\node[vertex] (d) at (1,0) {};
\node[vertex,label=above:$v_1$] (f) at (-1,1) {};
\node[vertex] (g) at (-1.5,1) {};
\node[vertex] (h) at (-1.5,1.5) {};
\node[vertex,label=right:$v_2$] (i) at (-0.5,1) {};
\node[vertex] (j) at (-0.5,1.5) {};
\node[vertex] (k) at (-0.5,0.5) {};
\node[vertex] (l) at (1,-0.5) {};
\node[vertex,label=above right:$v_3$] (m) at (1,-1) {};
\node[vertex] (n) at (1,-1.5) {};
\node[vertex] (o) at (0,-1) {};
\node[vertex] (p) at (2,-1) {};

\draw[blue] (a) -- (b);
\draw[blue] (b) -- (c);
\draw[blue] (c) -- (d);
\draw[red] (b) -- (f);
\draw[red] (f) -- (g);
\draw[red] (g) -- (h);
\draw[red] (f) -- (i);
\draw[red] (i) -- (j);
\draw[red] (i) -- (k);
\draw[green] (d) -- (l);
\draw[green] (l) -- (m);
\draw[green] (m) -- (n);
\draw[green] (m) -- (o);
\draw[green] (m) -- (p);

\draw[->] ($(a)+(0,0.35)$) -- (a);
\draw[->] (d) -- ($(d)+(0,0.35)$);

\end{tikzpicture}
\end{center}

\begin{remark}
There is a direct correspondence between the notions of \emph{main idempotent branch} and \emph{idempotent branch} and Definitions~1 (prime idempotents) and~3 (idempotent components) of~\cite{[PS05]}. More precisely, each main idempotent branch corresponds to a prime idempotent in the sense of ~\cite[Definition 1]{[PS05]}, and each idempotent branch corresponds to an idempotent component as in ~\cite[Definition 3]{[PS05]}.
\end{remark}
    \end{enumerate}
\end{definition}

\section{Languages of representatives of an equivalence relation}

Motivated by the study of the conjugacy languages in groups introduced in \cite{[CHHR16]}, given an equivalence relation $\sim$ in a finitely generated semigroup $S$ with finite generating set $X$, we define the language of the shortest words representing elements in each equivalence class, as
\[
\sim\!\text{-}\Geo_X(S)=\bigl\{w\in X^*: \ell(w)=\min\{\ell(u): u\in  X^*,\ (u\pi)\sim (w\pi)\}\bigr\}.
\]
Observe that we can fix a total order on $X$ which induces the ShortLex order on $X^*$ and allows us to define the language of ShortLex geodesics on $S$, which we will denote by $\SL_X(S)$. We can analogously define, for any equivalence relation $\sim$ on $S$, the language $\sim\!\text{-}\SL_X(S)$.

\begin{lemma}\label{lemma_inclusao}
	Let $S$ be a finitely generated semigroup, with finite generating set $X$. If $\sim_1$ and $\sim_2$ are equivalence relations on $S$ such that $\sim_1\,\subseteq\, \sim_2$, then $\sim_2\!\text{-}\Geo_X(S)\subseteq\;\sim_1\!\text{-}\Geo_X(S)$.
\end{lemma}
\begin{proof}
	We just need to observe that, in the conditions of the lemma, if $w\in\; \sim_2\!\text{-}\Geo_X(S)$, then for any $u\in \tilde{X}^*$ such that $(u\pi)\sim_1 (w\pi)$, we have that $(u\pi)\sim_2 (w\pi)$, and so $\ell(w)\leq\ell(u)$. Hence $w\in\; \sim_1\!\text{-}\Geo_X(S)$. Therefore $\sim_2\!\text{-}\Geo_X(S)\subseteq\;\sim_1\!\text{-}\Geo_X(S)$, as intended.
\end{proof}

\begin{remark}
    If $\sim$ is the identity relation, then $\sim\!\text{-}\Geo_X(S)=\Geo_X(S)$ is the language of geodesics on $X$.
\end{remark}

\begin{proposition}
    Let $S$ be a finitely generated monoid, with finite generating set $X$. If, on input $x,y\in X^*$, it is decidable whether $x\pi\sim y\pi$, then $\sim\!\text{-}\Geo_X(S)$ is recursive.
\end{proposition}
\begin{proof}
    We can decide if a given word $w$ belongs to $\sim\!\text{-}\Geo_X(S)$ by enumerating all words shorter than $w$ and checking if they represent elements equivalent to $w\pi$. If we find such a shorter word, then $w\not\in\; \sim\!\text{-}\Geo_X(S)$, and $w\in\;\sim\!\text{-}\Geo_X(S)$ otherwise.
\end{proof}

In this paper, we focus primarily on the study of Green's relations on a monoid, introducing the associated Green languages:

\begin{equation*}
	\begin{aligned}
    	\JGeo_X(S)&=\bigl\{w\in X^*: \ell(w)=\min\{\ell(u): u\in  X^*,\ (u\pi)\mathcal{J}(w\pi)\}\bigr\}\\
		\DGeo_X(S)&=\bigl\{w\in X^*: \ell(w)=\min\{\ell(u): u\in  X^*,\ (u\pi)\mathcal{D}(w\pi)\}\bigr\}\\
		\RGeo_X(S)&=\bigl\{w\in X^*: \ell(w)=\min\{\ell(u): u\in  X^*,\ (u\pi)\mathcal{R}(w\pi)\}\bigr\}\\
		\LGeo_X(S)&=\bigl\{w\in X^*: \ell(w)=\min\{\ell(u): u\in  X^*,\ (u\pi)\mathcal{L}(w\pi)\}\bigr\}\\
		\HGeo_X(S)&=\bigl\{w\in X^*: \ell(w)=\min\{\ell(u): u\in X^*,\  (u\pi)\mathcal{H}(w\pi)\}\bigr\}
	\end{aligned}
\end{equation*}

In this paper, we will focus in the case where $S$ is a free inverse monoid with the standard generating set. For this reason, since $\mathcal{D}=\mathcal{J}$, we will not state our results on  $\JGeo(S)_X$, mentioning only  $\DGeo(S)_X$. When the generating set $X$ of the free inverse monoid $\FIM_X$ is explicit, we will write $\sim\!\text{-}\Geo$ instead of $\sim\!\text{-}\Geo(\FIM_X)$, where $\sim$ will represent one of the Green's relations.

\section{A description of the Green Languages}
In this section, we present a visual description of the words that belong to each of the Green languages on free inverse monoids, in order to facilitate future proofs and justifications.

Let \(u = e_1 u_1 \ldots u_n e_{n+1}\) be a decomposition of a geodesic in $\FIM_X$, where $u_1 u_2 \cdots u_n$ is a reduced word when projected to the free group, $e_i \in E(\FIM)\setminus\{1\}$ for $i \in \{2, \ldots, n\}$, and $e_1, e_{n+1} \in E(\FIM)$.

\begin{enumerate}
	\item[G1] The word \(u_1 \ldots u_n\) represents a simple path (if we consider the Munn tree representing this reduced word, for each new letter, a new edge is added, we never go back along the walk, there are no branches, and no vertices are repeated).
	\item[G2] In the subtree determined by \(u_1 \ldots u_n\), we cross each edge exactly once.
	\item[G3] When we walk along a path determined by an idempotent, we start and finish at the same vertex (in fact, this is an equivalent condition for being an idempotent).
 	\item[G4] As a consequence of the previous point, the first letter of \(u_1 \ldots u_n\) is represented by an edge that starts at the initial vertex, and the last letter finishes at the final vertex.
	\item[G5] The walk represented by \(u_1 \ldots u_n\) in $MT(u)$ is the simple path (and hence the shortest one, since \(MT(u)\) is a tree) between the initial and final vertices, which is the \emph{main path} by Definition \ref{def_main/idempotent}.
	\item[G6] If the walk determined by a word crosses the same edge three times, then it is not a geodesic (by Lemma~\ref{passar3x}).
	\item[G7] The walk determined by an idempotent crosses each edge at least twice (we are working with trees).
	\item[G8] By 6.\,+\,7., the walk determined by an idempotent crosses each edge exactly twice.
    \item[G9] By 2.\,+\,8.\,+\,6., no edge of the main path of \(MT(u)\) can represent a letter of an idempotent of a decomposition of \(u\) (otherwise, we would cross that edge three times).
	\item[G10] Hence, in the walk determined by \(u\) in $MT(u)$, we cross each edge of the main path exactly once and each idempotent edge exactly twice (and thus the walk represented by \(u_1 \ldots u_n\) and the walk represented by the idempotents are disjoint walks, when viewed in terms of edges).
\end{enumerate}

The result below is proved in~\cite{[PS05]}. However, we include here a more geometric proof based on Munn trees, following a line of reasoning that is more closely aligned with the proofs of the main results of this article.

Throughout the paper, for the sake of greater fluency in the exposition, we will identify geodesics with their corresponding elements in the free inverse monoid. In particular, we will refer to a word representing an idempotent simply as an \emph{(geodesic) idempotent word} and for any geodesic words $u$ and $v$, we will write $u\LL v$ in stead of $u\pi \LL v\pi$, where $\pi$ is the natural onto morphism between $\tilde X^*$ and $\FIM_X$. Moreover, for any $g\in \FIM_X$, by the Munn tree representing $g$ we mean the Munn tree representing any word $u\in \tilde X^*$ such that $u\pi=g$. 

\begin{lemma}\label{lemma_dec_geo}
	Let $w$ be a geodesic in the free inverse monoid generated by a finite set $X$. Then there exists a unique decomposition
	$$
	w=e_1u_1e_2u_2\cdots e_nu_ne_{n+1},
	$$ 
	where $u_1u_2\cdots u_n$ is the reduced word when considering the projection to the free group, $e_i\in E(\FIM)\setminus \{1\}$ for $i\in\{2,\ldots, n\}$ and $e_1,e_{n+1}\in E(\FIM)$.
\end{lemma}
\begin{proof}
    Suppose that there exists another decomposition of $w$ satisfying the conditions of the statement. That is, we may write
\[
w = f_1 v_1 \ldots f_m v_m f_{m+1},
\]
where $v_1v_2\ldots v_m$ is the reduced word when considering the projection to the free group, $f_i\in E(\FIM)\setminus \{1\}$ for $i\in\{2,\ldots, m\}$ and $f_1,f_{m+1}\in E(\FIM)$.

Firstly, note that a geodesic word defines a walk (of minimal length) in the Munn tree that it represents, which is totally independent of the decomposition that we consider.

Further, observe that $v_1 \ldots v_m = u_1 \ldots u_n$, represents the (unique) simple path between the initial and the final vertex, i.e, the letters in $v_1 \ldots v_m = u_1 \ldots u_n$ are in (order preserving) bijection with the labels of the edges between the initial and final vertices, and when we read the letters of \( w \) that appear in \( u_1 \cdots u_n \) (and in \( v_1 \cdots v_m \)), the corresponding walk in the Munn tree must cross the corresponding edges of the main path of the Munn tree.

Now, if the two decompositions of the same word are distinct, then there exists a letter, say \(a\), that in one decomposition belongs to the reduced word (say to $u_i$, for some $i$) and in the other it belongs to an idempotent word (which has to be $f_i$ or $f_{i+1}$, since it is the first letter where the two decompositions differ). Because this occurrence of $a$ belongs to $u_i$, it implies that in $MT(u)$, when reading this letter, we must be crossing the edge of the main path corresponding to it.
However, since  $u_1 \ldots u_n = v_1 \ldots v_m$, and since this occurrence of $a$ belongs to an idempotent subword in the second decomposition, there must exist another occurrence of $a$ in the word $w$ that is the first letter of $v_i$ (if the first occurrence belongs to $f_i$) or the first letter of $v_{i+1}$ (if the first occurrence belongs to $f_{i+1}$). But this would imply that in $w$ there are two letters representing the same letter of the reduced word, which would mean that we have to cross the same edge of the main path twice. This observation contradicts $G10$, and so, the uniqueness of such a  decomposition is proved.
\end{proof}

Analogously to what we did for geodesics, we now provide a description of the elements of $\RGeo$ in terms of Munn trees, which will serve as a reference for future arguments. Observe that by Theorem \ref{Munn_Green} each $\RR$-class is uniquely determined by a tree with one rooted vertex (the initial vertex).
\begin{enumerate}
\item[R1] Given a tree with a fixed vertex (the initial vertex), the number of elements in the corresponding $\mathcal{R}$-class is equal to the number of vertices in the tree, corresponding to the possible choices for the final vertex.

\item[R2] Given a tree with the initial vertex fixed, by condition G10., a geodesic word representing such tree is shorter as the main path becomes longer, when we vary the final vertex.

\item[R3] By the definition of the main path (the unique simple path between the initial and final vertices), its length increases as the final vertex is placed further from the initial vertex.

\item[R4] A geodesic word is in $\RGeo$ if there is no vertex in its Munn tree whose distance from the initial vertex is strictly greater than the distance between the initial and final vertices. In other words, the main path is maximized.
\end{enumerate}

\begin{remark}
    By R4., no word in $\RGeo$ ends with an idempotent.
\end{remark}
Notice that an analogous argument can be made for $\LGeo$ by interchanging the terms \emph{initial} and \emph{final} vertices in the conditions stated above. We will refer to these conditions as L1--L4.

We also include below the explicit reasoning for the description of the elements of $\DGeo$ in terms of Munn trees. Notice that by Theorem \ref{Munn_Green}, each $\DD$-class is uniquely described by a tree with no rooted vertices.

\begin{enumerate}
    \item[D1] Given a tree with no rooted vertices (a representative of a $\DD$-class), the number of elements in this class is equal to the square of number of vertices in the tree, corresponding to the possible choices for the initial and final vertices.
    \item[D2] By condition G10., the geodesic word is shorter when the main path is longer.
    \item[D3] By definition of main path (see Definition \ref{def_main/idempotent}), its length increases as the rooted vertices are placed further from each other.
    \item[D4] A geodesic word is in $\DGeo$ if there is no pair of vertices in the tree whose unique simple path between them is greater than the length of the main path. 
\end{enumerate}

\begin{remark}
    By D4., no word in $\DGeo$ begins or ends with an idempotent.
\end{remark}

\section{Relations between the Green languages}
The purpose of this section is to establish relations between green languages. Concretely, we observe that $\mathcal{R}-\Geo$ and $\mathcal{L}-\Geo$ are contained in $\mathcal{H}-\Geo$, that $\mathcal{R}-\Geo$ is the inverse of  $\mathcal{L}-\Geo$ and that $\DGeo(\FIM_X)=\RGeo(\FIM_X)\cap \LGeo(\FIM_X),$ establishing a language-theoretic duel of the property that, in any semigroup, $\mathcal D=\mathcal{L}\vee \mathcal{R}.$
\begin{proposition}\label{inters_L_R}
	For any finite set $X$, we have 
	\[
	\DGeo(\FIM_X)=\RGeo(\FIM_X)\cap \LGeo(\FIM_X).
	\]
\end{proposition}
\begin{proof}
	Firstly, since $\mathcal{R},\mathcal{L}\subseteq \mathcal{D}$, we immediately have, by Lemma \ref{lemma_inclusao}, that
	\[
	\DGeo(\FIM_X)\subseteq\RGeo(\FIM_X) \cap \LGeo(\FIM_X).
	\]

	Now, by conditions D4, R4 and L4, a geodesic belongs to $\mathcal{D}\text{-}\Geo(\FIM_X),\ \mathcal{R}\text{-}\Geo(\FIM_X)$ or $\mathcal{L}\text{-}\Geo(\FIM_X)$ if the length of the main path is maximized.

    Given $w\in \RGeo(\FIM_X) \cap \LGeo(\FIM_X)$ we want to prove that there is no pair of vertices $(u,v)$ in $\Gamma(w)$ such that $d(u,v)>d(\alpha(w),\beta(w))$. 

    Suppose that such a pair exists. To prove that we always arrive to a contradiction, in the cases below we focus on the subtree generated by the vertices $\alpha:=\alpha(w)$, $\beta:=\beta(w)$, $u$ and $v$. We start by justifying that both $u$ and $v$ must belong to idempotent branches. In fact, if both $u$ and $v$ belong to the main path the contradiction is immediate. Assume now that one of the vertices, say $u$, belongs to the main path and the other belongs to an idempotent branch. Without loss of generality, assume that the subtree has the form below.    
 
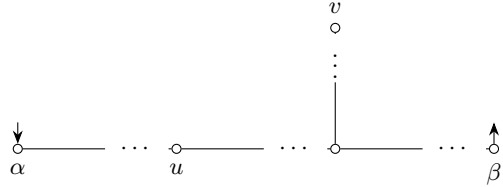
\begin{center}
\begin{tikzpicture}[
	vertex/.style={circle, draw, inner sep=1.2pt, fill=white},
	>={Stealth[length=5pt]},
	every node/.style={font=\small}
]

\node[vertex] (v0) at (0,0) {};

\node[vertex,label=below:$u$] (v1) at (-2.1,0) {};
\draw[-] (v1) -- ($(v1)!0.55!(v0)$);
\node at ($(v1)!0.75!(v0)$) {$\cdots$};
\draw[-] ($(v1)!0.95!(v0)$) -- (v0);

\node[vertex,label=below:$\alpha$] (v11) at (-4.2,0) {};
\draw[-] (v11) -- ($(v11)!0.55!(v1)$);
\node at ($(v11)!0.75!(v1)$) {$\cdots$};
\draw[-] ($(v11)!0.95!(v1)$) -- (v1);

\draw[->] ($(v11)+(0,0.35)$) -- (v11);

\node[vertex,label=above:$v$] (v2) at (0,1.6) {};
\draw[-] (v0) -- ($(v0)!0.55!(v2)$);
\node at ($(v0)!0.75!(v2)$) {$\vdots$};
\draw[-] ($(v0)!0.95!(v2)$) -- (v2);

\node[vertex,label=below:$\beta$] (v5) at (2.1,0) {};
\draw[-] (v0) -- ($(v0)!0.55!(v5)$);
\node at ($(v0)!0.75!(v5)$) {$\cdots$};
\draw[-] ($(v0)!0.95!(v5)$) -- (v5);

\draw[->] (v5) -- ++(0,0.35);

\end{tikzpicture}
\end{center}

Notice that in this case, $d(\alpha,v)\geq d(u,v)> d(\alpha, \beta)$, which contradicts the fact that $w\in \RGeo$.

    Therefore, $u$ and $v$ must belong to idempotent branches and we consider the 3 possible cases. 

    Firstly, assume that $u$ and $v$ are in the same main idempotent branch and that the simple path between $\alpha$ and $v$ contains $u$ (the case where the simple path between $\alpha$ and $u$ contains $v$ is clearly analogous).

    
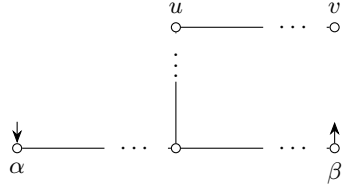
\begin{center}
\begin{tikzpicture}[
	vertex/.style={circle, draw, inner sep=1.2pt, fill=white},
	>={Stealth[length=5pt]},
	every node/.style={font=\small}
]

\node[vertex] (v0) at (0,0) {};

\node[vertex,label=below:$\alpha$] (v1) at (-2.1,0) {};
\draw[-] (v1) -- ($(v1)!0.55!(v0)$);
\node at ($(v1)!0.75!(v0)$) {$\cdots$};
\draw[-] ($(v1)!0.95!(v0)$) -- (v0);

\draw[->] ($(v1)+(0,0.35)$) -- (v1);

\node[vertex,label=above:$u$] (v2) at (0,1.6) {};
\draw[-] (v0) -- ($(v0)!0.55!(v2)$);
\node at ($(v0)!0.75!(v2)$) {$\vdots$};
\draw[-] ($(v0)!0.95!(v2)$) -- (v2);

\node[vertex,label=above:$v$] (v3) at (2.1,1.6) {};
\draw[-] (v2) -- ($(v2)!0.55!(v3)$);
\node at ($(v2)!0.75!(v3)$) {$\cdots$};
\draw[-] ($(v2)!0.95!(v3)$) -- (v3);

\node[vertex,label=below:$\beta$] (v5) at (2.1,0) {};
\draw[-] (v0) -- ($(v0)!0.55!(v5)$);
\node at ($(v0)!0.75!(v5)$) {$\cdots$};
\draw[-] ($(v0)!0.95!(v5)$) -- (v5);

\draw[->] (v5) -- ++(0,0.35);

\end{tikzpicture}
\end{center}

    Then it is immediate that if $d(u,v)>d(\alpha,\beta)$ then $d(\alpha,v)>d(\alpha,\beta)$ which contradicts the fact that $w\in \RGeo(\FIM_X)$.

    Next, assume that $u$ and $v$ belong to the same main idempotent branch but neither $u$ belongs to the simple path between $\alpha$ and $v$ nor $v$ belongs to the simple path between $\alpha$ and $u$.

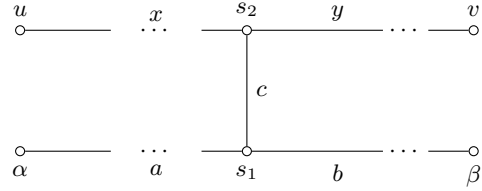
    \begin{center}
\begin{tikzpicture}[
	vertex/.style={circle, draw, inner sep=1.2pt, fill=white},
	every node/.style={font=\small}
]

\node[vertex,label=below:$\alpha$] (a0) at (-3,0) {};
\node[vertex,label=below:$s_1$] (a1) at (0,0) {};
\node[vertex,label=below:$\beta$] (a2) at (3,0) {};

\draw (a0) -- (-1.8,0);
\node at (-1.2,0) {$\cdots$};
\draw (-0.6,0) -- (a1);

\draw (a1) -- (1.8,0);
\node at (2.1,0) {$\cdots$};
\draw (2.4,0) -- (a2);

\node[below] at (-1.2,-0.05) {$a$};
\node[below] at (1.2,-0.05) {$b$};

\node[vertex,label=above:$u$] (u) at (-3,1.6) {};
\node[vertex,label=above:$s_2$] (m) at (0,1.6) {};
\node[vertex,label=above:$v$] (v) at (3,1.6) {};

\draw (u) -- (-1.8,1.6);
\node[above] at (-1.2,1.6) {$x$};
\draw (-0.6,1.6) -- (m);

\draw (m) -- (1.8,1.6);
\node[above] at (1.2,1.6) {$y$};
\draw (2.4,1.6) -- (v);

\node at (-1.2,1.6) {$\cdots$};
\node at (2.1,1.6) {$\cdots$};

\draw (m) -- node[right] {$c$} (a1);
\end{tikzpicture}
\end{center}

As represented in the scheme above we define $x=d(u,s_1)$, $y=d(s_2,v)$, $a=d(\alpha,s_1)$ and $b=d(s_1,\beta)$.

If $x+y = d(u,v)>d(\alpha,\beta)=a+b$, then we must have $x>a$ or $y>b$ (otherwise, we would have $x+y\leq a+b$).

Now, if $x>a$, then $d(u,\beta)=x+c+b>a+b=d(\alpha,\beta)$, which contradicts the fact that $w\in \LGeo$. If $y>b$, then $d(\alpha, v)=a+c+y>a+b=d(\alpha,\beta)$, which contradicts the fact that $w\in \RGeo$. 

Finally, assume that $u$ and $v$ belong to different main idempotent branches. Without loss of generality, assume that $u$ and $v$ are positioned as in the scheme below:
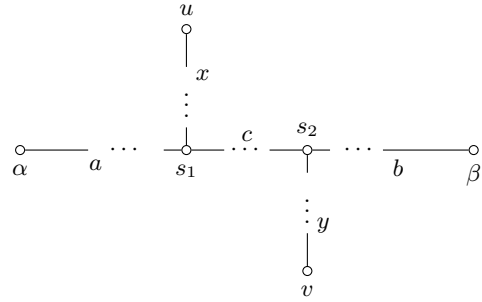
\begin{center}
\begin{tikzpicture}[
	vertex/.style={circle, draw, inner sep=1.2pt, fill=white},
	every node/.style={font=\small}
]

\node[vertex,label=below:$\alpha$] (a) at (-3,0) {};
\node[vertex,label=below:$s_1$] (m1) at (-0.8,0) {};
\node[vertex,label=above:$s_2$] (m2) at (0.8,0) {};
\node[vertex,label=below:$\beta$] (b) at (3,0) {};

\draw (a) -- (-2.1,0);
\draw (m1) -- (-0.8,0.3);
\draw (m2) -- (0.8,-0.3);
\node at (-1.6,0) {$\cdots$};
\draw (-1.1,0) -- (m1);

\draw (m1) -- (-0.3,0);
\draw (0.3,0) -- (m2);
\node at (0,0) {$\cdots$};

\draw (m2) -- (1.1,0);
\node at (1.5,0) {$\cdots$};
\draw (1.8,0) -- (b);

\node[vertex,label=above:$u$] (u) at (-0.8,1.6) {};

\node at (-0.8,0.7) {$\vdots$};
\draw (-0.8,1.1) -- (u);

\node[vertex,label=below:$v$] (v) at (0.8,-1.6) {};

\node at (0.8,-0.7) {$\vdots$};
\draw (0.8,-1.1) -- (v);

\node[below] at (-2,0) {$a$};
\node[above] at (0,0) {$c$};
\node[below] at (2,0) {$b$};
\node[right] at (-0.8,1) {$x$};
\node[right] at (0.8,-1) {$y$};

\end{tikzpicture}
\end{center}

As in the previous case, we define $a=d(\alpha,s_1)$, $b=d(\beta,s_2)$, $c=d(s_1,s_2)$, $x=d(s_1,u)$ and $y=d(s_2,v)$, and observe that if $d(u,v)>d(\alpha,\beta)$ we must have $x>a$ or $y>b$. The argument follows analogously.
	
	Hence,
	$$
	\mathcal{D}\text{-}\Geo(\FIM_X)=\mathcal{R}\text{-}\Geo(\FIM_X)\cap \mathcal{L}\text{-}\Geo(\FIM_X),
	$$
	as intended.
\end{proof}

\begin{remark}\label{RL_union}
	Notice that, for any finite set $X$, we always have
	$$
	\RGeo(\FIM_X)\cup \LGeo(\FIM_X) \subset \HGeo(\FIM_X).
	$$
	The inclusion follows immediately from Lemma \ref{lemma_inclusao}. 
	To justify that equality never occurs, it suffices to consider $a \in X$ and observe that the word $a^2 a^{-4} a$ belongs to $\mathcal{H}\text{-}\Geo(\FIM_X)$, but not to $\mathcal{L}\text{-}\Geo(\FIM_X)$, and also not to $\mathcal{R}\text{-}\Geo(\FIM_X)$.
\end{remark}

\begin{proposition}\label{RGeo_LGeo}
	For any finite set $X$ we have
$$
\RGeo(\FIM_X)=\LGeo(\FIM_X)^{-1}.
$$	
\end{proposition}
\begin{proof}
For any word \(u\), the words \(u\) and \(u^{-1}\) have the same tree, with the initial and final vertices interchanged. Therefore, if \(u \in \RGeo\) and \(v \,\LL\, u^{-1}\), 
then \(u^{-1}\) and \(v\) have the same tree and the same final vertex. Consequently, \(u\) and \(v^{-1}\) have the same tree and the same initial vertex, so \(v^{-1} \,\RR\, u\). This implies that \(\ell(u) \leq \ell(v^{-1})\), and hence \(\ell(u^{-1}) \leq \ell(v)\). Therefore, \(u^{-1}\in \LGeo(\FIM_X)\). The other inclusion is analogous.
\end{proof}

\section{Geodesics of Free Inverse Monoids}
In this section, we begin the study of the complexity of Green languages in free inverse monoids, starting by the relation $\mathcal{H}$. Since $\mathcal{H}$ is the trivial relation in a free inverse monoids, $\HGeo_X$ coincides with the language of geodesics with respect to the generating set $X$.

\begin{theorem}\label{geo_CF}
	Let $\FIM_X$ be the free inverse monoid generated by $X$. Then $\HGeo(\FIM_X)=\Geo(\FIM_X)$ is a context-free language.
\end{theorem}
\begin{proof}
Consider the following grammar with starting symbol $S$.
\begin{equation*}
	\begin{array}{lrlll}
		& S &\to &  E_{\tilde X\setminus\{x\}}xS_x \mid E_{\tilde X},\ & (x\in \tilde X) \\
		\text{for each } x\in \tilde X, & S_x & \to & E_{\tilde X\setminus\{x^{-1},y\}}yS_y \mid  E_{\tilde X\setminus \{x^{-1}\}} \mid \varepsilon,\ & (y\in \tilde X\setminus\{x^{-1}\}) \\
		\text{for each }  M\subseteq\tilde X,& E_M & \to & xE_{\tilde X\setminus\{x^{-1}\}}x^{-1} E_{M\setminus\{x\}} \mid \varepsilon,\ & (x\in M).
	\end{array}
\end{equation*}

Notice that in this grammar from $E_{\emptyset}$ we can only deduce $\varepsilon$.

Intuitively, the idea behind this grammar is to use the symbols \( E_M \) to generate idempotents and the symbols \( S_x \) to generate the reduced word, and at the same time “marking” the positions where the idempotents can be inserted. 
Moreover, the purpose of the productions involving the symbols \( E_M \), which give rise to two new symbols of the same form, is as follows: the first symbol is used to construct within the idempotent branch whose previous edge is already determined (represented above by \( x \)), while the second symbol is used to construct a new idempotent branch whenever necessary.

Firstly, let us show that any word $w$ that can be deduced from $E_M$, for any $M\subseteq \tilde X$, represents an idempotent. In fact, if $w$ is the empty word, the intended is immediate. By induction on the number of deductions, assume that for any word requiring at most $k$ deductions from $E_M$ represents an idempotent and that  $w$ requires $k+1$ deductions. For the first production we have $E_M\to xE_{\tilde X\setminus \{x^{-1}\}}x^{-1}E_{M\setminus \{x\}}$ for some $x\in M$. Then, since $E_{\tilde X\setminus \{x^{-1}\}}$ and $E_{M\setminus \{x\}}$ will need $k$ or less productions, by induction hypothesis they will produce idempotents, say $e$ and $f$ respectively. Hence $w=xex^{-1}f$, and so, $w$ is an idempotent. 

Further, note that every non trivial derivation from a symbol \( E_M \) produces an expression of the form
$x E_{\tilde{X} \setminus \{x^{-1}\}} x^{-1} \, E_{M \setminus \{x\}},\ x \in M,\; M \subseteq \tilde{X}.$
It follows that \( E_{\tilde{X} \setminus \{x^{-1}\}} \) already lies within an idempotent branch whose previous edge is labelled by \( x \in M \), while \( E_{M \setminus \{x\}} \) gives rise to a new branch, which must again start with a letter in \( M \). Consequently, every idempotent branch originating from a symbol \( E_M \) must begin with an edge whose label belongs to \( M \).

Consequently, we can now conclude that from $E_M$ one can deduce idempotents whose first letter in each of their main idempotent branches always lies in $M$.

Let $L$ be the context-free language defined by this grammar. We claim that $L=\Geo(\FIM_X)$.

Also, notice that any word that can be deduced from $S_x$, for any $x\in \tilde X$, then it can be deduced from the starting symbol $S$, as the productions from $S_x$ are the same as those of $S$ with more restrictions on the letters.

Now, consider $w\in L$. We prove that $w\in \Geo(\FIM_X)$ by induction on the number $k$ of deductions. If $k=2$, then $w=\varepsilon$ and $w\in \Geo(\FIM_X)$. Assume that the intended is verified for any word that requires at most $k\geq 2$ deductions and that $w$ requires $k+1$ deductions. Then we have the following possible scenarios of the first deductions of $w$.
\begin{itemize}
	\item[(A)] $S\to E_{\tilde X}\to xE_{\tilde X\setminus\{x^{-1}\}}x^{-1} E_{\tilde X\setminus\{x\}}\to ^+ w$
	\item[(B)] $S\to E_{\tilde X\setminus\{x\}}xS_x \to ^+ w$
\end{itemize}

In case $(A)$, we can write $w=xex^{-1}f$, where $e$ and $f$ are idempotents which are in $ \Geo(\FIM_X)$ by induction hypothesis. 

By definition of the grammar, $e$ is an idempotent for which none of its main idempotent branches start with $x^{-1}$ (because $x^{-1}\notin \tilde X\setminus \{x^{-1}\}$), and so $xex^{-1}$ is also a geodesic.

Further, since none of the  main branches of $f$ can start with $x$ (because $x\notin \tilde X\setminus \{x\}$), we get that $xex^{-1}f$ is also in $ \Geo(\FIM_X)$.

 Finally, in case $(B)$ we can write $w=exw'$, with $e,w'\in \Geo(\FIM_X)$, by induction hypothesis. By construction, none of the main idempotent branches of $e$ starts with $x$, and so $ex$ is also in $ \Geo(\FIM_X)$. Again by construction, no main idempotent branch of $e$ and no first edge of the main path of $w'$ can start with $x^{-1}$ and so $exw'\in \Geo(\FIM_X)$. 

Hence $L\subseteq \Geo(\FIM_X)$.

Now, we prove that $\Geo(\FIM_X)\subseteq L$ in three steps.

Firstly, notice that any reduced word $w$ is in $L$. In fact, if $w=x_1\cdots x_n$ is a reduced word, then we can consider the following deduction:

\begin{equation*}
	\begin{aligned}
		S&\to E_{\tilde X\setminus\{x_1\}}x_1S_{x_1} \to E_{\tilde X\setminus\{x_1\}}x_1E_{\tilde X\setminus\{x_1^{-1},x_2\}}x_2S_{x_2}\\
		&\to E_{\tilde X\setminus\{x_1\}}x_1E_{\tilde X\setminus\{x_1^{-1},x_2\}}x_2E_{\tilde X\setminus\{x_2^{-1},x_3\}}x_3S_{x_3}\\
		&\to^* E_{\tilde X\setminus\{x_1\}}x_1E_{\tilde X\setminus\{x_1^{-1},x_2\}}x_2E_{\tilde X\setminus\{x_2^{-1},x_3\}}x_3\ldots E_{\tilde X\setminus\{x_{n-1}^{-1},x_n\}}x_nS_{x_n}\to^+ w,
	\end{aligned}
\end{equation*}
where the final $\to^{+}$ represents the deductions $E_M \to \varepsilon$ for all symbols $E_M$, and $S_{x_n} \to \varepsilon$.

Next, we claim that any idempotent word $e\in \Geo(\FIM_X)$ is in $L$. We argue by induction on the length of $e$ (recall that $e$ has necessarily even length). If $e$ is the empty word or if $|e|=2$ then it is immediate that $e\in L$. Now, assume that $|e|=2k+2$ and that the result is valid for any $e$ of length at most $2k$.  We can write $e=e_1\ldots e_n$, where each $e_i$ represents one main idempotent branch.
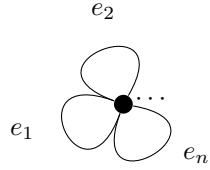
\begin{center}
	\begin{tikzpicture}[scale=1.3]

		\node[circle, fill=black, inner sep=2.5pt] (v) at (0,0) {};

		\draw (v) edge[loop left, min distance=1.1cm, in=250, out=160]
		node[left=8pt] {$e_1$} (v);

		\draw (v) edge[loop above, min distance=1.1cm, in=160, out=60]
		node[above=8pt] {$e_2$} (v);

		\draw (v) edge[loop left, min distance=1.1cm, in=350, out=250]
		node[right=8pt] {$e_n$} (v);
		
		\node at (0.3,0.05) {$\cdots$};
		
	\end{tikzpicture}
\end{center}

If $n=1$, then $e$ as only one branch and $e=xfx^{-1}$, with $f\in L$ by induction hypothesis. Hence, there is a sequence of deductions
$$
S\to E_{\tilde X} \to^+ f.
$$

Since $e$ is a geodesic, none of the idempotent branches of $f$ start with $x^{-1}$, and so $f$ can also be deduced from $E_{\tilde X\setminus \{x^{-1}\}}$. Hence, we have
$$
S\to E_{\tilde X} \to x E_{\tilde X\setminus \{x^{-1}\}} x^{-1} E_{\tilde X\{x\}} \to E_{\tilde X} \to x E_{\tilde X\setminus \{x^{-1}\}} x^{-1} \varepsilon \to^+ xfx^{-1} =e,
$$
and so, $e\in L$.

Now, if $n>1$, then by induction hypothesis, each $e_i$ is in $L$. Since $e$ is a geodesic, each branch, i.e., each $e_i$, start with different letters, say $x_1,\ldots, x_n$, and so, any $e_i$ can be deduced from $E_{\tilde X\setminus \{x_j: \text{ for all } j\neq i\}}$. Hence, we can consider the deductions
\begin{equation*}
	\begin{aligned}
		S&\to E_{\tilde X} \to x_1 E_{\tilde X\setminus\{x_1^{-1}\}} x_1^{-1} E_{\tilde X\setminus \{x_1\}} \\
		&\to x_1 E_{\tilde X\setminus\{x_1^{-1}\}}x_1^{-1} x_2 E_{\tilde X\setminus\{x_2^{-1}\}} x_2^{-1} E_{\tilde X\setminus \{x_1,x_2\}}\to^+ e_1\ldots e_n=e
	\end{aligned}
\end{equation*}

Finally, consider $w\in \Geo(\FIM_X)$, by Lemma \ref{lemma_dec_geo}, there is a unique decomposition 
	$$
u=e_1u_1e_2u_2\ldots e_nu_ne_{n+1},
$$ 
where $u_1u_2\cdots u_n$ is the reduced word when considering the image of $u$ to the free group, $e_i$ represent an idempotent in $\FIM_X$ for $i\in\{1,\ldots, n+1\}$, where $e_1,e_{n+1}$ are the only allowed to be trivial.
By the previous arguments, all the subwords $u_i$ and $e_i$, for all $i$, are in $L$. Let $x_i$ and $y_i$ denote the first and last letters of $u_i$, respectively, for every $i$. Since $u$ is a geodesic, it follows that for each $i$, no component of $e_i$ cannot begin with $x_i^{-1}$ or $y_{i+1}$. Therefore, $e_i$ can be derived from $E_{\tilde X \setminus \{x_i^{-1},, y_{i+1}\}}$, and hence we may consider the following sequence of derivations.
\begin{equation*}
	\begin{aligned}
		S&\to E_{\tilde X\setminus \{x_1\}}x_1S_{x_1}\to^+ e_1x_1S_{x_1}\to^+ e_1u_1S_{y_1}\to e_1u_1E_{\tilde X\setminus \{y_1^{-1},x_2\}}x_2S_{x_2} \\
		&\to^+ e_1u_1e_2x_2S_{x_2} \to^+ u
	\end{aligned}
\end{equation*}
which implies that $u\in L$.

Hence $L=\Geo(\FIM_X)$.
\end{proof}

\begin{theorem}\label{Sl_Geo}
    Let $\FIM_X$ be the free inverse monoid generated by $X$. Then $\mathcal{H}\text{-}\SL(\FIM_X)=\SL(\FIM_X)$ is a context-free language.
\end{theorem}
\begin{proof}
    The proof follows as of Theorem \ref{geo_CF} with the grammar
    \begin{equation*}
	\begin{array}{lrlll}
		& S &\to &  E_{\tilde X\setminus\{x\}}xS_x \mid E_{\tilde X},\ & (x\in \tilde X) \\
		\text{for each } x\in \tilde X, & S_x & \to & E_{\tilde X\setminus\{x^{-1},y\}}yS_y \mid  E_{\tilde X\setminus \{x^{-1}\}} \mid \varepsilon,\ & (y\in \tilde X\setminus\{x^{-1}\}) \\
		\text{for each } M\subseteq\tilde X,& E_M & \to & xE_{\tilde X\setminus\{x^{-1}\}}x^{-1} E_{\{y\in M:y>x\}} \mid \varepsilon,\ & (x\in M).
	\end{array}
\end{equation*}
The only modification required for the grammar of $\Geo(\FIM_X)$ 
is ensuring that each idempotent component is arranged in the correct order, i.e., respecting the ShortLex order.
\end{proof}

\begin{theorem} \label{thm_co-cont}
	Let $\FIM_X$ be the free inverse monoid generated by $X$. Then $\HGeo(\FIM_X)=\Geo(\FIM_X)$ is a co-context-free language.
\end{theorem}
\begin{proof}
	Recall that $\HGeo=\Geo(\FIM_X)$, since $\HH$ is trivial.
	Let $F_X$ be the free group generated by $X$ and $L=\WP(F_X)$ be its word problem. By the Muller-Schupp Theorem, $L$ is context-free \cite{[MS83]}. Notice that $L$ is the language of all words that represent an idempotent of $\FIM_X$.
	
	We prove that
	$$
	\tilde X^*\setminus \Geo(\FIM_X) = \bigcup_{a\in \tilde X} \tilde X^* aLa^{-1}La\tilde X^*.
	$$
Notice that, in terms of Munn trees, this equality asserts that a word is not geodesic if and only if it crosses some edge of the Munn tree three times.

The fact that any word that crosses an edge 3 times is not a geodesic is proved in Lemma \ref{passar3x}. We prove the converse.

	Consider $w\in \tilde X^*\setminus \Geo(\FIM_X)$. Then there exists $u\in \Geo(\FIM_X)$ such that $w\sim u$ and $\ell(u)<\ell(w)$. We have that $MT(u)=MT(w)$.

    By Fact~10, the word $u$ crosses each edge of the main path exactly once and each idempotent edge exactly twice. Since $\ell(u) < \ell(w)$, and each letter corresponds to the crossing of an edge, while $w$ has the same Munn tree, it follows that $w$ must cross at least one edge more times than $u$ does.

If this edge lies in an idempotent branch, then it follows immediately that $w$ crosses that edge three times.

Suppose now that the edge lies on the main path. In that case, there exists an edge of the main path that $w$ crosses at least twice. To complete the argument, observe that for any edge of the main path, the first crossing must occur in the direction of the final vertex. Since a Munn tree contains no cycles, the second crossing must occur in the opposite direction. Therefore, in order for $w$ to end at the final vertex, it must cross that same edge once more in the original direction, hence crossing it a third time, as required.
\end{proof}

\section{\texorpdfstring{$\DGeo(\FIM_X)$}{DGeo(\FIM)} on 2 or more generators}

In this section we study the language $\DGeo(\FIM_X)$ on two or more generators, while the monogenic case is studied in Section \ref{Monogenico}.

\begin{theorem} \label{D-CF}
    Let $X$ be a finite set with $\abs X\ge 2$, and $L\subseteq\DGeo(\FIM_X)$ a language containing at least one representative for each $\DD$-class. Then $L$ is not context-free. In particular, for any total order on $\tilde X$, $\DGeo(\FIM_X)$ and $\DD\text{-}\SL(\FIM_X)$ are not context-free.
\end{theorem}
\begin{proof}
Let $L\subseteq \DGeo(\FIM_X)$ be a language containing at least one element of each $\DD$-class. Consider
\[
K := L \cap \Bigl( a^+b^+(b^{-1})^+ \cup b^+(b^{-1})^+(a^{-1})^+\Bigr)
\]

Note that a word of the form $a^{n_1}b^{n_2}(b^{-1})^{n_3}$ with $n_1,n_2,n_3>0$ is in $\DGeo(\FIM_X)$ if and only if $n_1\ge n_2$ and $n_3\ge 2n_2$. Indeed,
\begin{itemize}[leftmargin=6mm]
    \item If $n_2>n_3$, then $a^{n_1}b^{n_2}\,\DD\, a^{n_1}b^{n_2}b^{-n_3}$, so that $a^{n_1}b^{n_2}b^{-n_3}\notin\DGeo(FIM_X)$.
    \item If $n_2\le n_3$ then we have a tree below and the word is in $\DGeo(FIM_X)$ if and only if the branch of length $n_2$ (the one crossed twice) is the shortest, i.e., $n_1\ge n_2$ and $n_3-n_2\ge n_2$.
\end{itemize} 
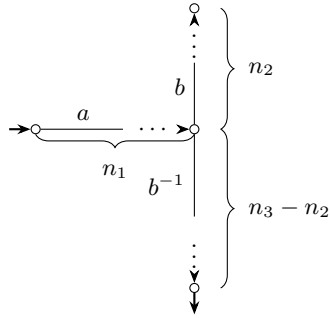
\begin{center}
	\begin{tikzpicture}[
		vertex/.style={circle, draw, inner sep=1.2pt, fill=white},
		>={Stealth[length=5pt]},
		every node/.style={font=\small}
		]

		\node[vertex] (v0) at (0,0) {};

		\node[vertex] (v1) at (-2.1,0) {};
		
		\draw[-] (v1) -- ($(v1)!0.55!(v0)$);
		\node at ($(v1)!0.75!(v0)$) {$\cdots$};
		\draw[->] ($(v1)!0.95!(v0)$) -- (v0);
		
		\node[above] at ($(v1)!0.3!(v0)$) {$a$};
		
		\draw[decorate,decoration={brace,mirror,amplitude=5pt}]
		(v1.south) -- (v0.south)
		node[midway,below=8pt] {$n_1$};

		\node[vertex] (v2) at (0,1.6) {};
		
		\draw[-] (v0) -- ($(v0)!0.55!(v2)$);
		\node at ($(v0)!0.75!(v2)$) {$\vdots$};
		\draw[->] ($(v0)!0.95!(v2)$) -- (v2);
		
		\node[left] at ($(v0)!0.35!(v2)$) {$b$};
		
		\draw[decorate,decoration={brace,mirror,amplitude=5pt}]
		([xshift=8pt]v0.east) -- ([xshift=8pt]v2.east)
		node[midway,right=7pt] {$n_2$};

		\node[vertex] (v3) at (0,-2.1) {};
		
		\draw[-] (v0) -- ($(v0)!0.55!(v3)$);
		\node at ($(v0)!0.75!(v3)$) {$\vdots$};
		\draw[->] ($(v0)!0.95!(v3)$) -- (v3);
		
		\node[left] at ($(v0)!0.35!(v3)$) {$b^{-1}$};
		
		\draw[decorate,decoration={brace,mirror,amplitude=5pt}]
		([xshift=8pt]v3.east) -- ([xshift=8pt]v0.east)
		node[midway,right=7pt] {$n_3-n_2$};

		\draw[->, thick] ($(v1)+(-0.35,0)$) -- (v1);
		\draw[->, thick] (v3) -- ++(0,-0.35);
		
	\end{tikzpicture}
\end{center}
A similar argument holds for words $b^{n_3}(b^{-1})^{n_2}(a^{-1})^{n_1}$. It follows that
\[ K \subseteq \bigl\{ a^{n_1}b^{n_2}(b^{-1})^{n_3},\; b^{n_3}(b^{-1})^{n_2}(a^{-1})^{n_1} \;\big|\; n_1,n_2,n_3\in\Z_{>0},\, n_1\ge n_2,\, n_3\ge 2n_2 \bigr\}. \]
Moreover, if $n_1>n_2$ and $n_3>2n_2$, then the two words are the only words in $\DGeo(\FIM_X)$ representing their $\DD$-class, hence at least one of the two belongs to $K$. 

We now prove that $L$ is not context-free. Suppose that $L$, and therefore $K$, are context-free. Let $p$ be the constant of the Ogden's Lemma for $K$ and assume, without loss of generality that, that $w=a^{p+1}b^p(b^{-1})^{2p+1}\in K$. Mark every position of the subword $b^p$ and consider the decomposition $w=uvxyz$ given by the Ogden's Lemma. We have that for all $i \geq 0$, the word $uv^i x y^i z$ belongs to $K$, so, since $v$ and/or $y$ must contain a $b$, it must be a power of $b$ (otherwise, after pumping, we would leave $K$). Hence, $v$ and $y$ cannot both contain $a$ and $b^{-1}$. But then, pumping increases the power $n_2$, while leaving $n_1$ and/or $n_3$ unchanged, which implies that, for a sufficiently large $i$, we will get $n_2>n_1$ or $n_3>2n_2$. This is a contradiction.
   \end{proof}

\begin{theorem}\label{DGeo_not_CCF}
    Let $X$ be a finite set such that $\abs X\ge 2$, and fix a total order on $\tilde X$. Let $L$ be a language satisfying $\DD\text{-}\SL(\FIM_X)\subseteq L\subseteq \DGeo(\FIM_X)$. Then $L$ is not co-context-free. In particular, $\DGeo(\FIM_X)$ and $\DD\text{-}\SL(\FIM_X)$ are not co-context-free.
\end{theorem}
\begin{proof}
    Consider $a,b\in\tilde X$ such that $a<a^{-1}$ and $b\in \widetilde X\setminus\{a,a^{-1}\}$. Let $K=L^c\cap a^+b^+(b^{-1})^+a^+$. 
    
    Suppose that $L$ is co-context-free, hence $K$ is context-free.
    Consider $p$ the constant of the Ogden's lemma and the word $w = a^p b^{p+1} b^{-(p+2)} a^{p+3},$ which is not in $\mathcal{D}$-Geo by D4, and so it belongs to $K$. Indeed, the distance between the initial and final vertices of the Munn tree of $w$ is
$
p+1+(p+3)=2p+4,$ 
and the distance between the upper $b$-leaf and the final vertex is
$
(p+1)+1+(p+3)=2p+5.
$

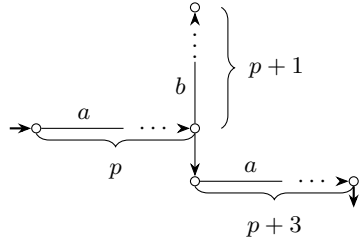
\begin{center}
	\begin{tikzpicture}[
		vertex/.style={circle, draw, inner sep=1.2pt, fill=white},
		>={Stealth[length=5pt]},
		every node/.style={font=\small}
		]
		
		\node[vertex] (v0) at (0,0) {};

		\node[vertex] (v1) at (-2.1,0) {};
		
		\draw[-] (v1) -- ($(v1)!0.55!(v0)$);
		\node at ($(v1)!0.75!(v0)$) {$\cdots$};
		\draw[->] ($(v1)!0.95!(v0)$) -- (v0);
		
		\node[above] at ($(v1)!0.3!(v0)$) {$a$};
		
		\draw[decorate,decoration={brace,mirror,amplitude=5pt}]
		(v1.south) -- (v0.south)
		node[midway,below=8pt] {$p$};
		
		\node[vertex] (v2) at (0,1.6) {};
		
		\draw[-] (v0) -- ($(v0)!0.55!(v2)$);
		\node at ($(v0)!0.75!(v2)$) {$\vdots$};
		\draw[->] ($(v0)!0.95!(v2)$) -- (v2);
		
		\node[left] at ($(v0)!0.35!(v2)$) {$b$};
		
		\draw[decorate,decoration={brace,mirror,amplitude=5pt}]
		([xshift=8pt]v0.east) -- ([xshift=8pt]v2.east)
		node[midway,right=7pt] {$p+1$};
	
		\node[vertex] (v3) at (0,-0.7) {};
		
			\draw[->] (v0) -- (v3);

		\node[vertex] (v4) at (2.1,-0.7) {};
		
		\draw[-] (v3) -- ($(v3)!0.55!(v4)$);
		\node at ($(v3)!0.75!(v4)$) {$\cdots$};
		\draw[->] ($(v3)!0.95!(v4)$) -- (v4);
		
		\node[above] at ($(v3)!0.35!(v4)$) {$a$};
		
		\draw[decorate,decoration={brace,mirror,amplitude=5pt}]
		(v3.south) -- (v4.south)
		node[midway,below=8pt] {$p+3$};

		\draw[->, thick] ($(v1)+(-0.35,0)$) -- (v1);
		\draw[->, thick] (v4) -- ++(0,-0.35);
		
	\end{tikzpicture}
\end{center}

We mark the first $p$ occurrences of the letter $a$. Let $w = uvxyz$ be a decomposition satisfying the hypotheses of the lemma. It is immediate that both $v$ and $y$ must be powers of a single letter. Moreover, at least one of them must be a power of part of the first $p$ occurrences of $a$.

We distinguish the following cases.

\begin{enumerate}
	\item $v = a^s$ and $y = a^\ell$, with $s \geq 1$ or $\ell \geq 1$.
	In this case, any pumping with exponent greater than \(2\) forces the resulting word to belong to \(\DD\text{-}\SL(\FIM_X)\) (and so not to belong to $K$), since the rooted vertices became the furthest possible and is the shortest with respect to the total order, as the only two words in $\mathcal D$-Geo representing the same $\mathcal{D}$-class as $w$ are $w$ and $w^{-1}$ but $a<a^{-1}$.
		
	\item $v = a^s$ and $y = b^\ell$.  
	In this case, if $k$ is the pumping parameter, for sufficiently large $k$, we get the word $w_k=uv^kxy^kz=a^{p+(k-1)s}b^{p+1+(k-1)\ell}b^{-(p+2)}a^{p+3}$ and the following Munn tree 
    
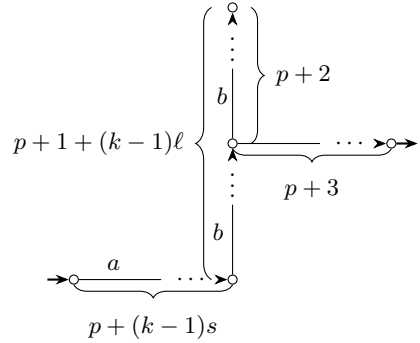
\begin{center}
			\begin{tikzpicture}[
				vertex/.style={circle, draw, inner sep=1.2pt, fill=white},
				>={Stealth[length=5pt]},
				every node/.style={font=\small}
				]

				\node[vertex] (v0) at (0,-1.8) {};
				\node[vertex] (v1) at (0,0) {};
				\node[vertex] (v2) at (0,1.8) {};

				\draw[-] (v0) -- ($(v0)!0.55!(v1)$);
				\node at ($(v0)!0.75!(v1)$) {$\vdots$};
				\draw[->] ($(v0)!0.95!(v1)$) -- (v1);
				
				\node[left] at ($(v0)!0.35!(v1)$) {$b$};
				
				\draw[-] (v1) -- ($(v1)!0.55!(v2)$);
				\node at ($(v1)!0.75!(v2)$) {$\vdots$};
				\draw[->] ($(v1)!0.95!(v2)$) -- (v2);
				
				\node at ($(v1)!0.35!(v2) + (-4pt,0)$) {$b$};
				
				\draw[decorate,decoration={brace,amplitude=6pt}]
				([xshift=-6pt]v0.west) -- ([xshift=-6pt]v2.west)
				node[midway,left=7pt] {$p+1+(k-1)\ell$};

				\draw[decorate,decoration={brace,mirror,amplitude=5pt}]
				([xshift=5pt]v1.east) -- ([xshift=5pt]v2.east)
				node[midway,right=6pt] {$p+2$};

				\node[vertex] (v4) at (2.1,0) {};
				
				\draw[-] (v1) -- ($(v1)!0.55!(v4)$);
				\node at ($(v1)!0.75!(v4)$) {$\cdots$};
				\draw[->] ($(v1)!0.95!(v4)$) -- (v4);

                  \draw[->, thick] (v4)--($(v4)+(0.35,0)$);
				
				\draw[decorate,decoration={brace,mirror,amplitude=5pt}]
				(v1.south) -- (v4.south)
				node[midway,below=8pt] {$p+3$};

				\node[vertex] (v5) at (-2.1,-1.8) {};
				
				\draw[-] (v5) -- ($(v5)!0.55!(v0)$);
				\node at ($(v5)!0.75!(v0)$) {$\cdots$};
				\draw[->] ($(v5)!0.95!(v0)$) -- (v0);
				
				\node[above] at ($(v5)!0.25!(v0)$) {$a$};

                \draw[->, thick] ($(v5)+(-0.35,0)$) -- (v5);
				
				\draw[decorate,decoration={brace,mirror,amplitude=5pt}]
				(v5.south) -- (v0.south)
				node[midway,below=8pt] {$p+(k-1)s$};
				
			\end{tikzpicture}
		\end{center}
    which, by D4., implies that  the only two words in $\DGeo(\FIM_X)$ representing this
$\mathcal D$-class are $w_k$ and $w_k^{-1}$. Since $w_k$ begins with $a$ and
$w_k^{-1}$ begins with $a^{-1}$, and $a<a^{-1}$, we have
    that $w$ represents a word in $\DD\text{-}\SL(\FIM_X)$,
    and so a word not in $K$.
		
	\item $v = a^s$ and $y = b^{-\ell}$.  
This case is analogous to Case~\(1\).
\end{enumerate}
	
	Hence, $L$ is not a context-free language.
\end{proof}

\section{\texorpdfstring{$\RGeo$}{RGeo} and \texorpdfstring{$\LGeo$}{LGeo} on 2 or more generators}

Analogously to the previous section, we study here the languages $\RGeo(\FIM_X)$ and $\LGeo(\FIM_X)$ over alphabets with two or more generators, while the monogenic case is studied in Section~\ref{Monogenico}.

\begin{theorem}\label{RL_CF}
Let $X$ be a finite set such that $|X|\ge 2$ and $L\subseteq \LGeo(\FIM_X)$ (resp.\ $\RGeo$) a language containing at least one representative of each $\LL$-class (resp.\ $\RR$-class). Then the $L$ is not context-free. In particular, for any fixed total order on $\tilde X$, the languages $\LGeo(\FIM_X)$ and $\LL\text{-}\SL(\FIM_X)$ (resp.\ $\RR$) are not context-free.
\end{theorem}
\begin{proof}
    Let $L\subseteq\LGeo(\FIM_X)$ be a language containing at least one element by $\LL$-class. Consider 
    \[
    K = L \cap \Bigl(a^+b^+(b^{-1})^+b^+ \cup a^+(b^{-1})^+b^+(b^{-1})^+\Bigr).
    \]
    We now prove that $L$ is not context-free. Suppose that $L$ is context-free, therefore $K$ is context-free too. Let $p$ be the constant of the Ogden's Lemma for $K$. Observe that
    \[
    a^{p+1}b^p(b^{-1})^{2p}b^p \quad\text{and}\quad a^{p+1}(b^{-1})^pb^{2p}(b^{-1})^p
    \]
    are the only two geodesics representing elements of their class in $\LGeo(\FIM_X)$. Without loss of generality, one can assume that $w=a^{p+1}b^p(b^{-1})^{2p}b^p\in K$,  and  we mark the last $p$ $b$'s. 
    
    By Ogden's Lemma, there exists a decomposition $ w = uvxyz $ such that
$$
uv^{m} xy^{m} z \in L \quad \text{for all } m \ge 0.
$$

Since $ K \subseteq a^+b^+(b^{-1})^+b^+ \cup a^+(b^{-1})^+b^+(b^{-1})^+$, both $v$ and $y$ must consist of powers of a single letter.  Moreover, one of them must necessarily be a power of the final block of $b$'s.

Observe that any positive pumping in the final block of $b$'s that is not accompanied by positive pumping in the block of $b^{-1}$'s produces a non-geodesic word. Therefore, we must have
$$
v = b^{-k} \quad \text{and} \quad y = b^{\ell},
$$
for some $k, \ell >0$. Notice that, for a word in $a^{+} b^{+} (b^{-1})^{+} b^{+}$ to be geodesic, the total exponent of $b$ must be less than or equal to the exponent of $b^{-1}$.

Suppose first that $k > \ell$. Then, for $m = 0$, the resulting word is no longer geodesic, since the total exponent of $b$ becomes strictly greater than that of $b^{-1}$. Similarly, if $k < \ell$, the word fails to be geodesic for any $m > 2$. Hence we conclude that $k = \ell$.

We may therefore write
$$
w = u b^{-k} x b^{k} z,
$$
for some $k \in \mathbb{N}$.

Now observe that any pumping with $m > 2$ increases the length of the lower idempotent branch of the associated Munn tree, making it longer than $p+1$. However, once the lower branch has length greater than $p+1$, the word no longer belongs to $\mathcal{L}\text{-}\mathrm{Geo}$, since the main path is no longer maximal.
	
		
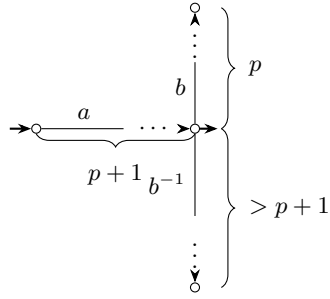
\begin{center}
		\begin{tikzpicture}[
			vertex/.style={circle, draw, inner sep=1.2pt, fill=white},
			>={Stealth[length=5pt]},
			every node/.style={font=\small}
			]

			\node[vertex] (v0) at (0,0) {};
			
			\node[vertex] (v1) at (-2.1,0) {};
			
			\draw[-] (v1) -- ($(v1)!0.55!(v0)$);
			\node at ($(v1)!0.75!(v0)$) {$\cdots$};
			\draw[->] ($(v1)!0.95!(v0)$) -- (v0);
			
			\node[above] at ($(v1)!0.3!(v0)$) {$a$};
			
			\draw[decorate,decoration={brace,mirror,amplitude=5pt}]
			(v1.south) -- (v0.south)
			node[midway,below=8pt] {$p+1$};

			\node[vertex] (v2) at (0,1.6) {};
			
			\draw[-] (v0) -- ($(v0)!0.55!(v2)$);
			\node at ($(v0)!0.75!(v2)$) {$\vdots$};
			\draw[->] ($(v0)!0.95!(v2)$) -- (v2);
			
			\node[left] at ($(v0)!0.35!(v2)$) {$b$};
			
			\draw[decorate,decoration={brace,mirror,amplitude=5pt}]
			([xshift=8pt]v0.east) -- ([xshift=8pt]v2.east)
			node[midway,right=7pt] {$p$};

			\node[vertex] (v3) at (0,-2.1) {};
			
			\draw[-] (v0) -- ($(v0)!0.55!(v3)$);
			\node at ($(v0)!0.75!(v3)$) {$\vdots$};
			\draw[->] ($(v0)!0.95!(v3)$) -- (v3);
			
			\node[left] at ($(v0)!0.35!(v3)$) {$b^{-1}$};
			
			\draw[decorate,decoration={brace,mirror,amplitude=5pt}]
			([xshift=8pt]v3.east) -- ([xshift=8pt]v0.east)
			node[midway,right=7pt] {$> p+1$};

			\draw[->, thick] ($(v1)+(-0.35,0)$) -- (v1);
			\draw[->, thick] (v0) -- ++(0.3,0);
			
		\end{tikzpicture}
	\end{center}

Therefore, Ogden's Lemma is not satisfied. It follows that $K$, and so $L$ are not context-free languages.
\end{proof}
\begin{proposition}\label{RL_CCF}
	Let $\FIM_X$ be the free inverse monoid generated by the finite set $X$ such that $|X|\ge 2$. Then $\RGeo(\FIM_X)$ and $\LGeo(\FIM_X)$ are not co-context-free languages.
\end{proposition}
\begin{proof}
	By Proposition \ref{RGeo_LGeo}, $\RGeo(\FIM_X)$ is co-context-free if and only if $\LGeo(\FIM_X)$ is co-context-free.
	
	Now, by proposition \ref{inters_L_R}, if $\RGeo(\FIM_X)$ and $\LGeo(\FIM_X)$ were co-context-free, so it would be $\DGeo(\FIM_X)$, which is impossible by Theorem \ref{DGeo_not_CCF}.
	
	Hence, $\RGeo(\FIM_X)$ and $\LGeo(\FIM_X)$ are not co-context-free languages.
\end{proof}

\section{Monogenic Free Inverse Monoid}\label{Monogenico}
Taking into account the results of this paper for generating sets of size at least two, it is now natural to seek a description of the Green languages in the monogenic case.  

\begin{theorem}\label{thm_monog}
	Let $\FIM_X$ be the monogenic free inverse monoid generated by $X=\{a\}$. 
    Then, we have the following:
	\begin{enumerate}
		\item $\DGeo(X)$ is a regular language;
		\item $\RGeo(X)$ is a deterministic context-free (and so also co-context-free) language;
		\item $\LGeo(X)$ is a deterministic context-free (and so also co-context-free) language;
		\item $\HGeo(X)$ is a deterministic context-free (and so also co-context-free) language.
	\end{enumerate}
\end{theorem}
\begin{proof}
    Firstly, recall that deterministic context-free languages are closed under complement. Hence, once proven that $\HGeo(\FIM_X)$, $\RGeo(\FIM_X)$ and $\LGeo(\FIM_X)$ are deterministic
context-free, it follows immediately that they are also co-context-free.

\medskip

\noindent $\bullet$ In the case of the \(\mathcal{R}\)-relation, each \(\mathcal{R}\)-class is uniquely determined by a chain of fixed length with a single rooted vertex (which corresponds to the initial vertex).
	
	
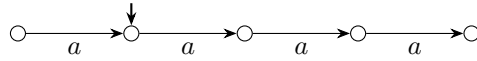
\begin{center}
		\begin{tikzpicture}[
			vertex/.style={circle, draw, inner sep=2pt}, 
			>={Stealth[length=5pt]}
			]
	
			\node[vertex] (e)  at (0,0) {};
			\node[vertex] (a1) at (1.5,0) {};
			\node[vertex] (a2) at (3,0) {};
			\node[vertex] (a3) at (4.5,0) {};
			\node[vertex] (a4) at (6,0) {};

			\draw[->] (e) -- node[below] {$a$} (a1);
			\draw[->] (a1) -- node[below] {$a$} (a2);
			\draw[->] (a2) -- node[below] {$a$} (a3);
			\draw[->] (a3) -- node[below] {$a$} (a4);
			
			\draw[->, thick] (1.5,0.4) -- (a1);
			
		\end{tikzpicture}
	\end{center}	
	
	A minimal length representative of a \(\mathcal{R}\)-class is given by the word obtained by moving along the entire chain in the shortest possible way, starting at the rooted vertex. If the initial vertex is a leaf, then it is immediate that the shortest way to follow the chain is to proceed directly to the opposite end. Otherwise, one must first travel to one of the ends and then reverse direction, going along the whole chain to the other end. In order to minimize the total distance, we should first choose the direction leading to the nearest leaf, minimizing the number of edges visited twice.
		
	
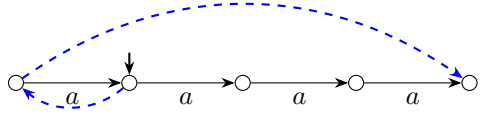
\begin{center}
		\begin{tikzpicture}[
			vertex/.style={circle, draw, inner sep=2pt},
			>={Stealth[length=5pt]}
			]
		
			\node[vertex] (e)  at (0,0) {};
			\node[vertex] (a1) at (1.5,0) {};
			\node[vertex] (a2) at (3,0) {};
			\node[vertex] (a3) at (4.5,0) {};
			\node[vertex] (a4) at (6,0) {};
		
			\draw[->] (e) -- node[below] {$a$} (a1);
			\draw[->] (a1) -- node[below] {$a$} (a2);
			\draw[->] (a2) -- node[below] {$a$} (a3);
			\draw[->] (a3) -- node[below] {$a$} (a4);
			
			\draw[->, thick] (1.5,0.4) -- (a1);
			
			\draw[->, blue, dashed, thick, bend left=35] (e) to (a4);
			\draw[->, blue, dashed, thick, bend left=40] (a1) to (e);
			
		\end{tikzpicture}
	\end{center}
	
	Therefore, we will cover twice at most half of the chain, leading to the inclusion
	$$
	\RGeo \subseteq \Big\{a^{-m}a^n \;\Big|\; m,n\in \N_0,\ m\leq \frac{n}{2}\Big\}\cup \Big\{a^ma^{-n} \;\Big|\; m,n\in \N_0,\ m\leq \frac{n}{2}\Big\}.
	$$
The other inclusion is immediate by the description of words in $\RGeo$ and that all the words in this language are geodesics (condition $G10$ is satisfied), hence we get the equality
\[
	\RGeo = \Big\{a^{-m}a^n \;\Big|\; m,n\in \N_0,\ m\leq \frac{n}{2}\Big\}\cup \Big\{a^ma^{-n} \;\Big|\; m,n\in \N_0,\ m\leq \frac{n}{2}\Big\}.
\]
Let us prove that the following deterministic pushdown automaton recognizes this language.
\begin{center}
	\begin{tikzpicture}[scale=.7, thick]	
		\node[state, accepting, initial left, initial distance=6mm, minimum size=15pt] (v0) at (0,0) {};
		\node[state, accepting, minimum size=15pt] (v1) at (4,2) {};
		\node[state, minimum size=15pt] (v2) at (9,2) {};
		\node[state, accepting, minimum size=15pt] (v3) at (14,2) {};
        
		\path (v0)  edge[-latex] node[sloped, above] {$(a,\#,x\#)$} (v1)
		(v1) edge [loop above, -latex] node [above] {$(a, x, xxx)$}  (v1)
        (v1) edge[-latex] node [above] {$(a^{-1}, x, \varepsilon)$} (v2)
		(v2) edge [loop above, -latex] node [above] {$(a^{-1}, x, \varepsilon)$} (v2)
		(v2) edge [-latex] node [above] {$(a^{-1}, \#, \#)$} (v3)
        (v3) edge [loop above, -latex] node [above] {$(a^{-1}, \#, \#)$} (v3);

        \node[state, accepting, minimum size=15pt] (u1) at (4,-2) {};
		\node[state, minimum size=15pt] (u2) at (9,-2) {};
		\node[state, accepting, minimum size=15pt] (u3) at (14,-2) {};
        
        \path (v0)  edge[-latex] node[sloped, below] {$(a^{-1},\#,x\#)$} (u1)
		(u1) edge [loop above, -latex] node [above] {$(a^{-1}, x, xxx)$}  (u1)
        (u1) edge[-latex] node [above] {$(a, x, \varepsilon)$} (u2)
		(u2) edge [loop above, -latex] node [above] {$(a, x, \varepsilon)$} (u2)
		(u2) edge [-latex] node [above] {$(a, \#, \#)$} (u3)
        (u3) edge [loop above, -latex] node [above] {$(a, \#, \#)$} (u3);
	\end{tikzpicture}
    \captionsetup{font=small}
    \captionof{figure}{A (one-counter) deterministic pushdown automaton recognizing $\RGeo$.}\label{fig:pushdown_R}
\end{center}

  Firstly observe that the initial and the first state of each branch are accepting states, and so the automaton accepts the empty word and all pure powers of $a$ and $a^{-1}$. These are precisely the words in $\RGeo$ obtained by taking $m=0$.

Focusing on the words accepted by the last state of the upper branch, after reading $a^m$, $m\geq 1$, the stack is $x^{2m-1}\#$. In fact, the first transition pushes one $x$, and the subsequent transition labelled by $a$ increases the number of $x$'s by two. Now, the automaton then reads $2m-1$ copies of $a^{-1}$, each popping one $x$, followed by one further $a^{-1}$ to enter the rightmost accepting state. The loop at that state permits any number of additional copies of $a^{-1}$. Thus the mixed words accepted along the upper branch are exactly
$$\Big\{a^ma^{-n} \;\Big|\; m,n\in \N_0,\ m\leq \frac{n}{2}\Big\}.$$

The lower branch is symmetric: it accepts every pure word $a^{-m}$, with $m\ge1$, and its mixed accepted words are exactly
\[
\Big\{ a^{-m}a^n \;\Big|\; m,n\in \N_0,\ m\leq \frac{n}{2} \Big\}.
\]
Combining the two branches and the empty word gives precisely $\RGeo$, as intended.

\medskip

\noindent $\bullet$ Regarding the $\mathcal{L}$-relation, by Proposition \ref{RGeo_LGeo}, we get that   	
		$$
	\LGeo = \Big\{a^{-n}a^m \;\Big|\; m,n\in \N_0,\ m\leq \frac{n}{2}\Big\}\cup \Big\{a^na^{-m} \;\Big|\; m,n\in \N_0,\ m\leq \frac{n}{2}\Big\}. 
	$$
Let us prove that the following deterministic pushdown automaton recognizes this language.

\begin{center}
	\begin{tikzpicture}[scale=.7, thick]	
		\node[state, accepting, initial above, initial distance=6mm, minimum size=15pt] (v0) at (0,2) {};
		\node[state, accepting, minimum size=15pt] (v1) at (4,2) {};
		\node[state, accepting, minimum size=15pt] (v2) at (9,2) {};
		\node[state, accepting, minimum size=15pt] (v3) at (6.5,-2) {};
        
		\path (v0)  edge[-latex] node[sloped, above] {$(a,\#,\#)$} (v1)
		(v1) edge [bend left, -latex] node [above, text width=45pt] {$(a, \#, x\#)$ $(a,x,xx)$}  (v2)
        (v2) edge[bend left, -latex] node[below] {$(a,x,x)$} (v1)
		(v1) edge[-latex] node[sloped, below] {$(a^{-1}, x, \varepsilon)$} (v3)
        (v2) edge[-latex] node[sloped, below] {$(a^{-1}, x, \varepsilon)$} (v3)
        (v3) edge[loop below,-latex] node[below] {$(a^{-1}, x, \varepsilon)$} (v3);

        \node[state, accepting, minimum size=15pt] (u1) at (-4,2) {};
		\node[state, accepting, minimum size=15pt] (u2) at (-9,2) {};
		\node[state, accepting, minimum size=15pt] (u3) at (-6.5,-2) {};
        
		\path (v0)  edge[-latex] node[sloped, above] {$(a^{-1},\#,\#)$} (u1)
		(u1) edge [bend right, -latex] node [above, text width=55pt] {$(a^{-1}, \#, x\#)$ $(a^{-1},x,xx)$}  (u2)
        (u2) edge[bend right, -latex] node[below] {$(a^{-1},x,x)$} (u1)
		(u1) edge[-latex] node[sloped, below] {$(a, x, \varepsilon)$} (u3)
        (u2) edge[-latex] node[sloped, below] {$(a, x, \varepsilon)$} (u3)
        (u3) edge[loop below, -latex] node[below] {$(a, x, \varepsilon)$} (u3);
	\end{tikzpicture}
    \captionsetup{font=small}
    \captionof{figure}{A (one-counter) deterministic pushdown automaton recognizing $\LGeo$.}\label{fig:pushdown_L}
\end{center}

Let $\mathcal A$ denote the automaton in the figure, with initial
stack symbol $\#$ and acceptance by final state. We show that
$L(\mathcal A)=\LGeo$. Since the initial state is accepting,
$\varepsilon \in L(\mathcal A)$.

Consider the right-hand branch. After reading $a^n$, where $n \geq 1$,
the stack is $x^{\lfloor n/2 \rfloor}\#$. Indeed, the first transition
leaves the stack unchanged. Subsequently, the automaton alternates
between its two upper states, pushing one $x$ upon reading each
even-numbered $a$ and leaving the stack unchanged upon reading each
odd-numbered $a$. Both states are accepting, so every word $a^n$
is accepted.

To read $a^{-1}$, the automaton must move to the lower state,
popping one $x$. Each subsequent $a^{-1}$ also pops one $x$,
and no transition is available when the stack top is $\#$.
Since the lower state is accepting, a word $a^n a^{-m}$,
with $n,m \geq 1$, is therefore accepted if and only if
$$
m \leq \ \frac{n}{2}.
$$
 Moreover, the transitions ensure
that no other form of word is accepted along this branch.

By symmetry, the nonempty words accepted along the left-hand
branch are precisely the words $a^{-n}a^m$ with $n \geq 1$
and $0 \leq m \leq n/2$. Together with the empty word, this yields
\[
L(\mathcal A)
=
\Big\{a^{-n}a^m \;\Big|\; m,n \in \mathbb N_0,\,
m \leq \frac{n}{2}\Big\}
\cup
\Big\{a^n a^{-m} \;\Big|\; m,n \in \mathbb N_0,\,
m \leq \frac{n}{2}\Big\}
=
\LGeo.
\]

\noindent $\bullet$ Regarding $\HGeo$, we already know that the language is context-free (Theorem \ref{geo_CF}) and co-context-free (Theorem \ref{thm_co-cont}). In the monogenic case, the language is given by
\[
\HGeo = \Bigl\{ a^ma^{-n}a^p \;\Big|\; m,n,p\in\N_0,\; m+p\le n \Bigr\} \cup \Bigl\{ a^{-m}a^n a^{-p} \;\Big|\; m,n,p\in\N_0,\; m+p\le n \Bigr\}.
\]

Indeed, observe that each \(\mathcal{H}\)-class is uniquely determined by a chain of fixed length with a two rooted vertices (corresponding to the initial and final vertices).

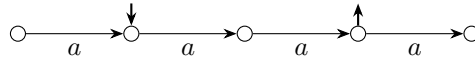
\begin{center}
		\begin{tikzpicture}[
			vertex/.style={circle, draw, inner sep=2pt}, 
			>={Stealth[length=5pt]}
			]

			\node[vertex] (e)  at (0,0) {};
			\node[vertex] (a1) at (1.5,0) {};
			\node[vertex] (a2) at (3,0) {};
			\node[vertex] (a3) at (4.5,0) {};
			\node[vertex] (a4) at (6,0) {};

			\draw[->] (e) -- node[below] {$a$} (a1);
			\draw[->] (a1) -- node[below] {$a$} (a2);
			\draw[->] (a2) -- node[below] {$a$} (a3);
			\draw[->] (a3) -- node[below] {$a$} (a4);

			\draw[->, thick] (1.5,0.4) -- (a1);
			\draw[->, thick] (a3) -- (4.5,0.4);
		\end{tikzpicture}
	\end{center}	

A minimal length representative of an \(\mathcal{H}\)-class is obtained as follows. Starting at the initial vertex, we move along the chain in the direction opposite to the final vertex until reaching the corresponding leaf, provided that the initial vertex is not itself a leaf. We then traverse the entire chain to the opposite leaf and, if this leaf is not the final vertex, move back along the chain until reaching the final vertex.

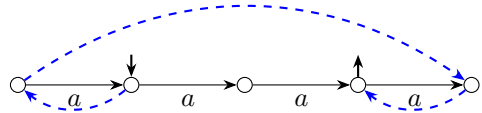
\begin{center}
		\begin{tikzpicture}[
			vertex/.style={circle, draw, inner sep=2pt},
			>={Stealth[length=5pt]}
			]

			\node[vertex] (e)  at (0,0) {};
			\node[vertex] (a1) at (1.5,0) {};
			\node[vertex] (a2) at (3,0) {};
			\node[vertex] (a3) at (4.5,0) {};
			\node[vertex] (a4) at (6,0) {};

			\draw[->] (e) -- node[below] {$a$} (a1);
			\draw[->] (a1) -- node[below] {$a$} (a2);
			\draw[->] (a2) -- node[below] {$a$} (a3);
			\draw[->] (a3) -- node[below] {$a$} (a4);

			\draw[->, thick] (1.5,0.4) -- (a1);
            \draw[->, thick] (a3) -- (4.5,0.4);

			\draw[->, blue, dashed, thick, bend left=35] (e) to (a4);
			\draw[->, blue, dashed, thick, bend left=40] (a1) to (e);
			\draw[->, blue, dashed, thick, bend left=40] (a4) to (a3);
		\end{tikzpicture}
	\end{center}

    This leads to the following conclusion.
\[
\HGeo\subseteq
\Bigl\{ a^ma^{-n}a^p \;\Big|\; m,n,p\in\N_0,\; m+p\le n \Bigr\} \cup \Bigl\{ a^{-m}a^n a^{-p} \;\Big|\; m,n,p\in\N_0,\; m+p\le n \Bigr\}.
\]

The other inclusion is immediate by the description of words in $\HGeo$ (see condition $G10$), hence we get the intended equality.
    
Let us prove that the following deterministic pushdown automaton recognizes this language.
\begin{center}
	\begin{tikzpicture}[scale=.6, thick]	
		\node[state, accepting, initial left, initial distance=6mm, minimum size=15pt] (v0) at (0,0) {};
		\node[state, accepting, minimum size=15pt] (v1) at (4,2) {};
		\node[state, accepting, minimum size=15pt] (v2) at (9,2) {};
		\node[state, accepting, minimum size=15pt] (v3) at (14,2) {};
        \node[state, accepting, minimum size=15pt] (v4) at (19,2) {};
        
		\path (v0)  edge[-latex] node[sloped, above] {$(a,\#,x\#)$} (v1)
		(v1) edge [loop above, -latex] node [above] {$(a, x, xx)$}  (v1)
        (v1) edge[-latex] node [above] {$(a^{-1}, x, \varepsilon)$} (v2)
		(v2) edge [loop above, -latex] node [above] {$(a^{-1}, x, \varepsilon)$} (v2)
		(v2) edge [-latex] node [above] {$(a^{-1}, \#, x\#)$} (v3)
        (v3) edge [loop above, -latex] node [above] {$(a^{-1}, x, xx)$} (v3)
        (v3) edge [-latex] node [above] {$(a, x, \varepsilon)$} (v4)
        (v4) edge [loop above, -latex] node [above] {$(a, x, \varepsilon)$} (v4);

        \node[state, accepting, minimum size=15pt] (u1) at (4,-2) {};
		\node[state, accepting, minimum size=15pt] (u2) at (9,-2) {};
		\node[state, accepting, minimum size=15pt] (u3) at (14,-2) {};
        \node[state, accepting, minimum size=15pt] (u4) at (19,-2) {};
        
        \path (v0)  edge[-latex] node[sloped, below] {$(a^{-1},\#,x\#)$} (u1)
		(u1) edge [loop above, -latex] node [above] {$(a^{-1}, x, xx)$}  (u1)
        (u1) edge[-latex] node [above] {$(a, x, \varepsilon)$} (u2)
		(u2) edge [loop above, -latex] node [above] {$(a, x, \varepsilon)$} (u2)
		(u2) edge [-latex] node [above] {$(a, \#, x\#)$} (u3)
        (u3) edge [loop above, -latex] node [above] {$(a, x, xx)$} (u3)
        (u3) edge [-latex] node [above] {$(a^{-1}, x, \varepsilon)$} (u4)
        (u4) edge [loop above, -latex] node [above] {$(a^{-1}, x, \varepsilon)$} (u4);
	\end{tikzpicture}
    \captionsetup{font=small}
    \captionof{figure}{A (one-counter) deterministic pushdown automaton recognizing $\HGeo$.}\label{fig:pushdown_H}
\end{center}
First the start state accepts $\varepsilon$. We next focus on the top branch. The first state accepts words of the form $a^m$ for $m>0$, the second accepts words $a^na^{-p}$ for $0<p\le n$, and the third accepts words $a^ma^{-n}$ for $0<m<n$. At that point, the stack reads $x^{n-m}\#$ and any additional letter $a$ read gets us onto the fourth state and removes one $x$, accepting all words $a^ma^{-n}a^p$ with $p\le n-m$.
The lower branch is analogous, exchanging the signs.

Therefore, this automaton accepts the words
\begin{equation*}
    \begin{aligned}
        \{\varepsilon\}&\cup\{a^m|\ m>0\}\cup \{a^na^{-p}|\ 0<p\le n\}\cup \{a^ma^{-n}|\ 0<m<n\} \cup \{a^ma^{-n}a^p|\ m+p\le n\} \\
        &\cup \{a^{-m}|\ m>0\}\cup \{a^{-n}a^{p}|\ 0<p\le n\}\cup \{a^{-m}a^{n}|\ 0<m<n\} \cup \{a^{-m}a^{n}a^{-p}|\ m+p\le n\},
    \end{aligned}
\end{equation*}
which clearly equals $\HGeo$.
\medskip

\noindent $\bullet$ Focusing now on the $\mathcal{D}$-classes, we restrict our attention to chains without rooted vertices. Consequently, we consider only chains of varying lengths. Each $\mathcal{D}$-class is uniquely determined by a chain of fixed length. \vspace*{2mm}
	
	\begin{center}
		\begin{tikzpicture}[
			vertex/.style={circle, draw, inner sep=2pt}, 
			>={Stealth[length=5pt]}
			]
			
			\node[vertex] (e)  at (0,0) {};
			\node[vertex] (a1) at (1.5,0) {};
			\node[vertex] (a2) at (3,0) {};
			\node[vertex] (a3) at (4.5,0) {};
			\node[vertex] (a4) at (6,0) {};
			
			\draw[->] (e) -- node[below] {$a$} (a1);
			\draw[->] (a1) -- node[below] {$a$} (a2);
			\draw[->] (a2) -- node[below] {$a$} (a3);
			\draw[->] (a3) -- node[below] {$a$} (a4);
			
		\end{tikzpicture}
	\end{center}
	
	For each such chain, the shortest way to go along the whole chain is to start at one of the leafs and proceed to the opposite end, yielding words of the form $a^{n}$ with $n \in \mathbb{Z}$. In this way, with the exception of the $\mathcal{D}$-class of the identity element, every $\mathcal{D}$-class has exactly two geodesics of minimal length, namely $a^{n}$ and $a^{-n}$ for some $n \in \mathbb{N}$. Hence,
	\[
	\DGeo= a^{*} \cup (a^{-1})^{*},
	\]
	which is a regular language.
\end{proof}

\begin{remark}
    The languages $\RGeo$, $\LGeo$ and $\HGeo$ described above are not regular. 

    \smallskip
    
    Indeed, suppose first that $\RGeo$ is regular, and let $p$ be its pumping length. Consider the word $w=a^{-p}a^{2p}\in \RGeo$. For any decomposition $w=xyz$ satisfying $|xy|\leq p$ and $|y|>0$, we have $y=a^{-k}$ for some $1\leq k\leq p$. Pumping $y$ once gives $xy^2z=a^{-(p+k)}a^{2p}$. For this word to belong to $\RGeo$, we would need $p+k\leq \frac{2p}{2}=p,$ which is impossible. Hence $xy^2z\notin\RGeo$, contradicting the pumping lemma. Therefore $\RGeo$ is not regular.

    \smallskip

    Since the class of regular languages is closed under reversal and $\LGeo=\RGeo^{-1}$ (Lemma \ref{RGeo_LGeo}), the result extends directly to $\LGeo$.

    \smallskip
    
    Finally, suppose that $\HGeo$ is regular, with pumping length $p$. Consider the word $w=a^pa^{-2p}a^p\in\HGeo$. Again, for any decomposition $w=xyz$ with $|xy|\le p$ and $|y|>0$, we have $y=a^{k}$ for some $1\le k\le p$, and therefore $xy^2z=a^{p+k}a^{-2p}a^p\notin \HGeo$ (since $2p<(p+k)+p$), contradicting the pumping lemma.
\end{remark}

\section{Growth tightness of Free Inverse Monoids}

In order to prove growth tightness of $\FIM_X$, we translate the question into an analogous question on languages where more is know (eg.\ \cite{Growth_sensitive}) under the name ``growth sensitivity".
\begin{definition}
	Given a language $L\subseteq \Sigma^*$ and $f\in \Sigma^*$, we define
	\begin{itemize}[leftmargin=5mm]
		\item $\Sub(L)$ is the set of subwords of $L$, i.e. $\Sub(L) = \bigl\{ v\in\Sigma^* \mid \exists u,w\in\Sigma^*, uvw\in L\bigr\}$, and
		\item $L^f = L\setminus \Sigma^*f\Sigma^*$ the set of words of $L$ which do not contain the forbidden subword $f$.
	\end{itemize}
	We also define some notations in relation to growth:
	\begin{itemize}[leftmargin=5mm]
		\item $\gamma_L(n)=\#(L\cap\Sigma^n)$ is the \emph{growth function} of $L$, and
		\item $\alpha_L=\limsup_{n\to\infty}\sqrt[n]{\gamma_L(n)}$ is the (exponential) \emph{growth rate} of $L$. 
	\end{itemize}
	A language is \emph{growth sensitive} if $\alpha(L)>\alpha(L^f)$ for all $f\in\Sub(L)$.
\end{definition}
The basic idea is to consider $L=\SL(\FIM_X)$ and prove that this language is growth sensitive. (However, we will not quite be able to prove this.) Suppose this were true, and consider any proper quotient $\pi\colon(\FIM_X,X)\to (Q,X)$. There exists $g_1,g_2\in\FIM_X$ such that $g_1\ne g_2$ in $\FIM_X$ but $\pi g_1=\pi g_2$ in $Q$. Consider $w_1,w_2\in\SL(\FIM_X)$ the ShortLex representatives of $g_1,g_2$. If $w_1<w_2$ in the ShortLex order, then $\SL(Q,X)\subseteq L^{w_2}$ hence
\[ \alpha(Q,X) = \alpha(\SL(Q,X)) \le \alpha(L^{w_2}) < \alpha(L) = \alpha(\FIM_X).\]

A general result in that direction is the following:
\begin{theorem}[{\cite[Theorem 2(B)]{Growth_sensitive}}] \label{thm:irred_growth_sensitive}
	Consider a language $L$ generated by an irreducible unambiguous context-free grammar, and $f\in\Sub(L)$. Then $\alpha(L^f) < \alpha(L)$.
\end{theorem}
However, the grammar producing $\SL(\FIM_X)$ is not irreducible. Instead, we will look into the associated growth series (as in the proof of Theorem \ref{thm:irred_growth_sensitive}) to prove our result. In particular, we will study the type of singularities of these series:
\begin{definition}
	Consider a formal series $\Gamma(z) = \sum_{n=0}^{\infty} a_n\cdot z^n$ with coefficients $a_n\ge 0$.
	\begin{itemize}[leftmargin=5mm]
		\item The \emph{growth rate} of its coefficients is $\alpha=\limsup_{n\to\infty}\sqrt[n]{a_n}$.
		\item The \emph{radius of convergence} is $\rho=1/\alpha$.
		\item The series is \emph{convergent} if $\Gamma(\rho)<\infty$, and \emph{divergent} if $\Gamma(\rho)=\infty$. 
        \item A point $z=z_0$ a \emph{singularity} of $\Gamma$ if $\Gamma(z)$ cannot be extended as an analytic function in a neighbourhood of $z_0$.
        \item A singularity at $z=z_0$ is called a \emph{square-root singularity} if there exist constants $A,B\in\mathbb{C}$ such that, as $z\to\rho$,
        \[
        \Gamma(z)=A + B\sqrt{1-\frac{z}{\rho}} + O\!\left(1-\frac{z}{\rho}\right).
        \]
	\end{itemize}
\end{definition}
We will use two results on singularities of series:
\begin{theorem}[{Drmota-Lalley-Woods, see \cite[Theorem VII.5]{Flajolet}}] \label{thm:DLW}
	Consider a language $L$ generated by an irreducible unambiguous context-free grammar. Then the growth series $\Gamma_L(z)$ has a square-root singularity at the radius of convergence $\rho$, that is,
	\[ \Gamma_L(z) = A + B \sqrt{1-\frac z\rho} + O\left(1-\frac z\rho\right) \quad\text{when }z\to\rho. \]
	In particular, the series converges at $z=\rho$.
\end{theorem}

Intuitively, the Drmota--Lalley--Woods theorem states that sufficiently well-behaved context-free languages have a universal asymptotic behaviour. Although the combinatorial structure of such languages may be complicated, their growth series always develops the same type of singularity at its dominant critical point, namely a square-root singularity. In other words, from the analytic point of view, all irreducible unambiguous context-free languages behave in essentially the same way near their radius of convergence.

\medbreak

The last result couples nicely with the following lemma:
\begin{lemma} \label{lem:common_radius}
    If $A,B$ are two variables in a common strongly connected component of the dependency graph, then $\Gamma_{L(A)}(z)$ and $\Gamma_{L(B)}(z)$ share the same radius of convergence.
\end{lemma}
\begin{proof}
    We have $A\to^* \alpha B\beta$ for some $\alpha,\beta\in \Sigma^*$, hence $L(A)\supseteq \alpha L(B)\beta$ and
    \[ \Gamma_{L(A)}(z) \ge z^{\abs \alpha+\abs \beta}\cdot \Gamma_{L(B)}(z),\]
    therefore $\rho_{L(A)}\le\rho_{L(B)}$. The reverse inequality holds for the same reason.
\end{proof}

\medskip

At the other end of the spectrum, many series diverges for easy reasons:
\begin{lemma} \label{lem:diverge}
	If $L$ is an infinite language that is closed under taking subwords, i.e.\ $uvw\in L\implies v\in L$ for all $u,v,w\in \Sigma^*$, then $\Gamma_{L}(z)$ diverges at the radius of convergence.
\end{lemma}
\begin{proof}
	Since $L$ is closed under subwords, we have $(L\cap \Sigma^{m+n}) \subseteq (L\cap\Sigma^m)(L\cap\Sigma^n)$ hence $\gamma_L(m+n)\le\gamma_L(m)\gamma_L(n)$: the growth function $\gamma_L$ is sub-multiplicative. By Fekete's lemma, it follows that the exponential growth rates exists and satisfies
	\[ \alpha(L) := \lim_{n\to\infty} \sqrt[n]{\gamma_L(n)} \overset!= \inf_{n\ge 1}\sqrt[n]{\gamma_L(n)}.\]
	(Moreover, this infimum is non-zero since $L$ is infinite.) It follows that $\gamma_L(n)\ge \alpha(L)^n$ hence
	\[ \Gamma_L(z) = \sum_{n=0}^\infty \gamma_L(n)\cdot z^n \ge \sum_{n=0}^\infty \bigl(\alpha(L)\cdot z\bigr)^n. \]
    Therefore, since the series on the right-hand side diverges at $z=\rho=1/\alpha(L)$, it follows that the series on the left-hand side also diverges at the same point.
\end{proof}
In particular, Lemma \ref{lem:diverge} applies to $L=\SL(Q,X)$ for any monoid $Q$ and any finite generating set $X$: growth series of infinite monoids diverge.

\bigskip

Finally, solutions of linear systems are monotonous as a function of the coefficients.
\begin{definition}[{Order on $\R[[z]]$}]
    We define $\sum_n a_nz^n \ge \sum_n b_nz^n$ if $a_n\ge b_n$ for all $n$. This order extends to vectors and matrices over $\R[[z]]$ if the inequalities hold component-wise.
\end{definition}
\begin{lemma} \label{lem:monotone}
    Consider two linear equations $\mathbf x = zA\cdot \mathbf x + \mathbf b$ and $\mathbf x=zA'\cdot\mathbf x+\mathbf b'$, where the coefficients of $A,\mathbf b,A',\mathbf b'$ are in $\R_{\ge 0}[[z]]$. If $A\le A'$ and $\mathbf b\le \mathbf b'$, then the solutions satisfy $$\mathbf x(A,\mathbf b)\le \mathbf x(A',\mathbf b').$$
\end{lemma}
\begin{proof}
    The solution is given by the explicit formula
    \[ \mathbf x(A,\mathbf b)=\mathbf b + (zA)\mathbf b + (zA)^2\mathbf b + \ldots \in \R_{\ge 0}[[z]], \]
    which is clearly increasing in $A$ and $\mathbf b$.
\end{proof}

\subsection{Proof}
The goal of this section is to prove the following result.

\begin{theorem}\label{growth}
    Let $M$ be an inverse monoid generated by a finite set $X$, with $\abs X\ge 2$. Then $\alpha(M,X) \le \alpha(\FIM_X,X)$ with equality if and only if $M=\FIM_X$.
\end{theorem}
To do so, we fix a total order on $\tilde X$. We define a  set (to get a \emph{reduced} grammar).
\begin{align*}
\mathcal P & = \Bigl\{ \bigl\{y\in\tilde X:y\ge x\bigr\}\setminus Y \;\Big|\; x\in\tilde X,\, Y\subset\tilde X,\, \abs Y\le 2\Bigr\} \setminus\{\emptyset\}, \\
\mathcal Q & = \Bigl\{ \bigl\{y\in\tilde X:y\ge x\bigr\}\setminus Y \;\Big|\; x\in\tilde X,\, Y\subset\tilde X,\, \abs Y\le 1\Bigr\} \setminus\{\emptyset\}.
\end{align*}
Recall that the following grammar unambiguously produces $\SL(\FIM_X,X)$:
\begin{equation*}
	\begin{array}{lrlll}
		& S &\to &  E_{\tilde X\setminus\{x\}}xS_x \mid E_{\tilde X},\ & (x\in \tilde X), \\
		\text{for each } x\in \tilde X, & S_x & \to & E_{\tilde X\setminus\{x^{-1},y\}}yS_y \mid  E_{\tilde X\setminus \{x^{-1}\}} \mid \varepsilon,\ & (y\in \tilde X\setminus\{x^{-1}\}), \\
		\text{for each } M\in\mathcal P,& E_M & \to & xE_{\tilde X\setminus\{x^{-1}\}}x^{-1} E_{\{y\in M:y>x\}} \mid \varepsilon,\ & (x\in M\text{ and }x\ne\max M), \\
        \text{for each } M\in\mathcal P,& E_M & \to & xE_{\tilde X\setminus\{x^{-1}\}}x^{-1} \mid \varepsilon,\ & (x\in M \text{ and }x=\max M).
	\end{array}
\end{equation*}
(This is essentially the grammar from  Theorem \ref{Sl_Geo}, however we removed all the variables that cannot be reached from $S$. We replaced $E_\emptyset$ by $\varepsilon$, since $E_\emptyset\to\varepsilon$ was the only production rule that could applied.) The dependency graph between these variables is as follows:
\begin{center}
	\begin{tikzpicture}[scale=.9, thick]	
		\node[state, minimum size=25pt] (S) at (0,5) {$S$};
		{\footnotesize
		\draw[blue, very thick] (0,0) circle (2.5cm);
		\node[state, minimum size=22pt] (Sx) at (1.5,0) {$S_x$};
		\node[state, minimum size=22pt, inner sep=0pt] (Sxx) at (-1.5,0) {$S_{x^{-1}}$};
		\node[state, minimum size=22pt] (Sy) at (0,1.5) {$S_y$};
		\node[state, minimum size=22pt, inner sep=0pt] (Syy) at (0,-1.5) {$S_{y^{-1}}$};}
		
		\node[state, minimum size=25pt, inner sep=0pt] (EX) at (8,5) {$E_{\tilde X}$};
        \node[state, minimum size=22pt] (EN1) at (4,-.7) {\footnotesize$E_M$};
        \node[state, minimum size=22pt] (EN2) at (3.4,.8) {};
        \node[state, minimum size=22pt] (EN3) at (4.7,.7) {};
        \node at (4,-1.7) {\footnotesize($M\in\mathcal P\setminus\mathcal Q$)};
        \path (2.5,-.5) edge[-latex, very thick] (EN1);
        \path (EN1) edge[-latex, very thick] (5.5,-.5);
        \path (2.5,.5) edge[-latex, very thick] (EN2);
        \path (EN2) edge[-latex, very thick] (EN3);
        \path (EN3) edge[-latex, very thick] (5.5,.5);
        
		\draw[red, very thick] (8,0) circle (2.5cm);
		\node[state, minimum size=25pt, inner sep=0pt] (EM) at (8,0) {$E_{M}$};
		\node[state, minimum size=25pt, inner sep=0pt] (EM1) at (6.5,1) {};
		\node[state, minimum size=25pt, inner sep=0pt] (EM2) at (9,1.5) {};
		\node[state, minimum size=25pt, inner sep=0pt] (EM3) at (7,-1.5) {};
		\node at (9,-1.2) {\footnotesize($M\in\mathcal Q\setminus\{\tilde X\}$)};
		
		\path (S) edge[-latex, very thick] (EX);
		
		\path (Sx) edge[-latex, bend right] (Sy);
		\path (Sy) edge[-latex, bend right] (Sxx);
		\path (Sxx) edge[-latex, bend right] (Syy);
		\path (Syy) edge[-latex, bend right] (Sx);
		
		\path (Sx) edge[-latex, bend right] (Syy);
		\path (Syy) edge[-latex, bend right] (Sxx);
		\path (Sxx) edge[-latex, bend right] (Sy);
		\path (Sy) edge[-latex, bend right] (Sx);
		
		\path (Sx) edge[loop right] (Sx);
		\path (Sy) edge[loop above] (Sy);
		\path (Sxx) edge[loop left] (Sxx);
		\path (Syy) edge[loop below] (Syy);
		
		\path (EM) edge[-latex, dotted] (EM1);
		\path (EM1) edge[-latex, dotted] (EM2);
		\path (EM2) edge[-latex, dotted] (EM);
		\path (EM1) edge[-latex, dotted, bend right] (EM3);
		\path (EM3) edge[-latex, dotted, bend right] (EM1);
		
		\path (S) edge[-latex, out=-135, in=100, very thick] (-1.5,2);
		\path (S) edge[-latex, very thick] (6,1.5);
		\path (EX) edge[-latex, out=-45, in=80, very thick] (9.5,2);
		\path (1.5,-2) edge[-latex, out=-30, in=-150, very thick] (6.5,-2);
	\end{tikzpicture}
	
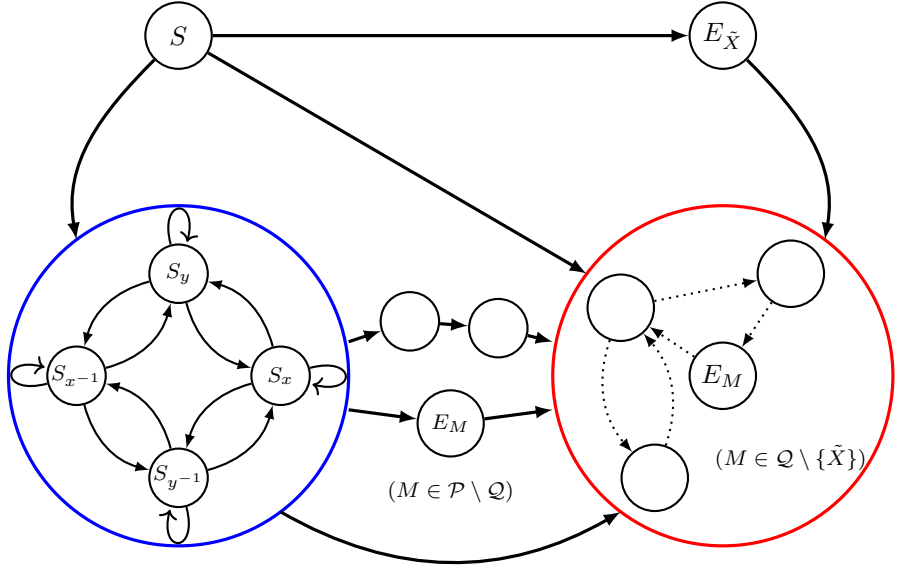
\captionof{figure}{The dependency graph, with its strongly connected components}
\end{center}
Let us make a few observations:
\begin{itemize}[leftmargin=6mm]
    \item The grammar formed by the variables $E_M$ with $M\in\mathcal Q\setminus\{\tilde X\}$ satisfies the conditions of Theorem \ref{thm:DLW} and Lemma \ref{lem:common_radius}. (Note that the starting variable does not matter.) It follows that the growth series $\Lambda_M(z)$ have a square root singularity at at their common radius $\rho_E$. In particular, the series $\Lambda_M(z)$ are convergent at $\rho_E$.
    \item We also observe that \vspace*{1mm}
\[ \Lambda_{M}(z) = 1 + \sum_{x\in M} z^2 \cdot \Lambda_{\tilde X\setminus\{x^{-1}\}}(z)\cdot \Lambda_{\{y\in M:y>x\}}(z), \vspace*{-1mm} \]
(with the convention $\Lambda_\emptyset=1$). By induction on $\abs M$, we can express each series $\Lambda_M$ (with $M\in\mathcal P\setminus \mathcal Q$ and $M=\tilde X$) as a polynomial in $z$ and $\Lambda_N$ (for $N\in\mathcal Q\setminus\{\tilde X\}$). It follows that \emph{all} the series $\Lambda_M$ (with $M\in\mathcal P$) share a common radius of convergence $\rho_E$, where they have a square root singularity, in particular converge.
\smallbreak
\end{itemize}
The other growth series $\Gamma_{\SL}(z)$ and $\Delta_x(z) := \Gamma_{L(S_x)}(z)$ satisfy a system of linear equations:

\begin{align}
	\Gamma_\SL(z) & = \Lambda_{\tilde X}(z) + \sum_{x\in\tilde X} z \cdot \Lambda_{\tilde X\setminus\{x\}}(z)\cdot \Delta_x(z), \label{eq: gamma}\\
	\Delta_x(z) & = 1 + \Lambda_{\tilde X\setminus\{x^{-1}\}}(z) + \sum_{y\ne x^{-1}} z\cdot  \Lambda_{\tilde X\setminus\{x^{-1},y\}}(z)\cdot \Delta_y(z),\label{eq: delta}	
\end{align}
 We obtain that $\Gamma_\SL(z) = r\bigl(z,(\Lambda_M(z))_M \bigr)$ where $r=\frac pq$ is a rational function. We can give more precise expressions for $p$ and $q$. We define the vectors
\[
\mathbf{\Delta(z)}
= (\Delta_x(z))_{x\in\widetilde X},
\quad
\mathbf{b}(z)
= \left( 1+ \Lambda_{\widetilde X\setminus\{x^{-1}\}}(z)
\right)_{x\in\widetilde X},
\quad\text{and}\quad 
\mathbf{c}(z)
= \left( \Lambda_{\widetilde X\setminus\{x\}}(z) \right)_{x\in\widetilde X}
\]
and the matrix $A(z)$, indexed by $\tilde X$, by
$$
A_{x,y}(z)
=
\begin{cases}
	\Lambda_{\widetilde X\setminus\{x^{-1},y\}}(z),
	& y\neq x^{-1},\\[1mm]
	0,  & y=x^{-1}.
\end{cases}
$$
Then Equations~\eqref{eq: gamma} and \eqref{eq: delta} become
\[
	\mathbf\Delta(z)
	=
	\mathbf b(z)+ z\,A(z)\mathbf\Delta(z)
	\qquad 
	\text{and}
	\qquad 
		\Gamma_\SL(z)
	=
	\Lambda_{\widetilde X}(z)
	+
	z\,\mathbf c(z)^{\mathsf T}\mathbf\Delta(z).
\]

The full system can be written as
$$
\begin{pmatrix}
	1 & -z\,\mathbf c(z)^{\mathsf T}\\
	0 & I-z\,A(z)
\end{pmatrix}
\begin{pmatrix}
	\Gamma_\SL(z)\\
	\mathbf\Delta(z)
\end{pmatrix}
=
\begin{pmatrix}
	\Lambda_{\widetilde X}(z)\\
	\mathbf b(z)
\end{pmatrix}.
$$
Thus, by Cramer's rule, $\Gamma_\SL(z)=\frac pq$ with 
\begin{align}
    p\bigl(z,(\Lambda_M)_M\bigr)
    & = \det\begin{pmatrix}
	\Lambda_{\tilde X}(z) & -z\,\mathbf c(z)^{\mathsf T}\\
	\mathbf b(z) & I-z\,A(z)
\end{pmatrix} \in \Z[z, \Lambda_M], \label{eq:numerator}\\[1mm]
    q\bigl(z,(\Lambda_M)_M\bigr)
	& = \det\bigl(I-z\,A(z)\bigr). 
\label{eq:denominator}
\end{align}
To continue the proof we need to introduce further notations and prove one lemma.
\begin{itemize}[leftmargin=6mm]
    \item For $x\in \tilde X$, we denote by $A_x\subset\FIM_X$ the set of idempotents with geodesic representatives of the form $xwx^{-1}$.
    \item For $r\ge 0$, we define $w_r:=\prod\{e: e\in E(\FIM_X),\ |e|=2r\}$.
    \end{itemize}
    
\begin{lemma} \label{lem:missing_idempotent}
	Consider the free inverse monoid $\FIM_X$ of rank $\abs X\ge 2$ and $x\in \tilde X$. In any proper quotient of $\FIM_X$ there exists distinct idempotents $e_1,e_2\in A_x$ such that $e_1\sim e_2$.
\end{lemma}
\begin{proof}
	Let $M$ be a proper quotient of $\FIM_X$. Then there exist two distinct elements $u_1,u_2\in \FIM_X$ such that $u_1\sim u_2$ in $M$. Recall that, for $g\in\FIM_X$, we denote $\bar g\in\FIM_X$ the associated reduced word / its ``main path''. We argue in two steps:

    \bigskip
    
    \noindent (1) There exist distinct elements $v_1,v_2\in\FIM_X$ with $\bar v_1=\bar v_2$ and $v_1\sim v_2$.
    \medskip
    
    \noindent We suppose that $\bar u_1\ne \bar u_2$. Take $r\ge 2\max\{\abs {u_1},\abs{u_2}\}$. Since $w_r(u_1\bar u_1^{-1})=w_r$, we have
    \begin{align*}
    u_1\bar u_1^{-1}\sim u_2\bar u_1^{-1} 
    & \,\implies 
    w_r(u_1\bar u_1^{-1})w_r(u_2\bar u_1^{-1}) \sim w_r(u_2\bar u_1^{-1})w_r(u_1\bar u_1^{-1}) \\
    & \iff 
    w_r(u_2\bar u_1^{-1}) \sim w_r(u_2\bar u_1^{-1})w_r.
    \end{align*}
    Note that $v_1:=w_r (u_2\bar u_1^{-1})\ne w_r(u_2\bar u_1^{-1}) w_r=:v_2$, since the Munn tree of the LHS is a ball of radius $r$, while the RHS is the union of two balls or radius $r$. Moreover 
    \[
    \overline{w_r(u_2\bar u_1^{-1}) w_r} 
    = \overline{w_r (u_2\bar u_1^{-1})} 
    = \bar u_2\bar u_1^{-1}.
    \]
   
   \noindent (2) There exists distinct idempotents $e_1,e_2\in A_x\subset E(\FIM_X)$ with $e_1\sim e_2$.
   \medskip
    
    \noindent Put $v:=\overline{v_1}=\overline{v_2}$. Then $f_1:=v_1v^{-1}$ and $f_2:=v_2v^{-1}$ are idempotent and are equal in the quotient. Moreover $f_1\ne f_2$, since $v_1\neq v_2$ and $\overline{v_1}=\overline{v_2}$, hence
    \[ \Gamma(f_1)=\Gamma(v_1)\neq \Gamma(v_2)=\Gamma(f_2). \]

\noindent Let $q$ be a vertex belonging to only one of $\Gamma(v_1)$ and $\Gamma(v_2)$ and $z$ be the first letter of $\overline{q}$ and $y\neq z$ be a letter. Let $w$ be a word such that $u=xwy^{-1}$ is reduced and $|u|>\max\{|f_1|,|f_2|\}$.

\noindent We check that $e_1=uf_1u^{-1}$ and $e_2=uf_2u^{-1}$ satisfy the claim:
\begin{itemize}[leftmargin=6mm]
    \item They are clearly idempotent and represent the same element in the quotient.
    \item Since $u$ is longer than $f_1$ and $f_2$, when we read these subwords after reading $u$, we do not get back to the initial vertex. So the only edge in $\Gamma(e_1),\Gamma(e_2)$ leaving the initial vertex is the one labelled by $x$, i.e., $e_1,e_2\in A_x$
    \item To see that $f_1\ne f_2$, assume w.l.o.g. that $q\in \Gamma(v_1)$. Then, the vertex $q'$ which is the corresponding to $q$ in the translation of $\Gamma(v_1)$ obtained after $u$ belongs to $\Gamma(e_1)$. Also, it belongs to the branch at $u$ that starts with $z$. We cannot reach it in $e_2$ because it does not belong to $\Gamma(uf_2)$ and we then follow the branch starting in $y$. \qedhere
\end{itemize}
\end{proof}
    
\begin{center}
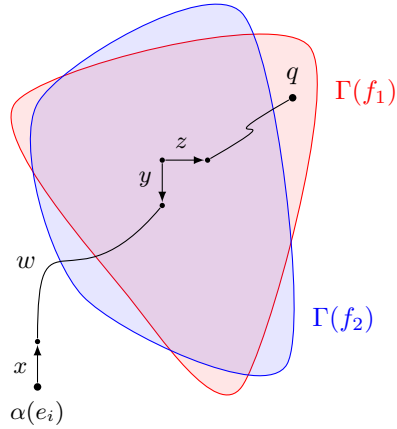

    \begin{tikzpicture}[scale=1.5]
        \draw[red, fill=red, fill opacity=.1] plot[smooth cycle] coordinates {(0, -.2) (1, -1) (1.6, 2) (-1, 1.5)};
        \draw[blue, fill=blue, fill opacity=.1] plot[smooth cycle] coordinates {(-.4, -.2) (1.4, -.8) (1, 2.3) (-.8, 1.5)};

        \node[circle, inner sep=1pt, fill=black, label=above:$q$] (q) at (1.45,1.55) {};
        \node[circle, inner sep=.7pt, fill=black] (e) at (.3,1) {};
        \node[circle, inner sep=.7pt, fill=black] (z) at (.7,1) {};
        \node[circle, inner sep=.7pt, fill=black] (y) at (.3,.6) {};
        \node[circle, inner sep=.7pt, fill=black] (x) at (-.8,-.6) {};
        \node[circle, inner sep=1pt, fill=black, label=below:\small$\alpha(e_i)$] (alpha) at (-.8,-1) {};
        \draw[-latex] (e) to (z);
        \node at (.47,1.15) {\small$z$};
        \draw (z) to[out=30,in=-150, looseness=3] (q);

        \draw[-latex] (e) to (y);
        \node at (.15,.83) {\small$y$};
        \draw (y) to[out=-130, in=90, looseness=2] (x);
        \node at (-.9,.1) {$w$};
        \draw[-latex] (alpha) -- (x);
        \node at (-.95,-.83) {\small$x$};
        
        \node[red] at (2.1,1.6) {$\Gamma(f_1)$};
        \node[blue] at (1.9,-.4) {$\Gamma(f_2)$};
        
    \end{tikzpicture}
    \captionof{figure}{Construction of $e_1,e_2$.}
\end{center}

Let $w_1,w_2$ be the ShortLex representatives of $e_1,e_2$ with $w_1<w_2$ in the ShortLex order. Then $f=w_2$ satisfies $f\in \Sub(L(E_M))$ for $M=\tilde X\setminus\{x^{-1}\}\in\mathcal Q\setminus\{\tilde X\}$ (hence for all $M\in\mathcal P$, using the dependency graph), and $f\notin\SL(Q,X)$. It follows that $\SL(Q,X)\subseteq \SL^f \subsetneq\SL$. The language $\SL^f$ is defined by a similar grammar: we have the same production rules for $S, S_x$, and then define infinitely many production rules
\[ E_M \to w  \qquad \bigl(w\in L(E_M)^f\bigr). \]
The key point is that $\Gamma_{\SL^f}(z) = r\bigl(z,(\Lambda^f_M(z))_M\bigr)$. Now we compare the radius of convergence of this new series with that of $\Gamma_\SL(z)$, understanding where the singularities come from:
\begin{itemize}[leftmargin=6mm]
    \item $\Gamma_{\SL}(z)$ diverges at $\rho_\SL$ (Lemma \ref{lem:diverge})
    \item the numerator $p$ is a polynomial function of $z$ and $\Lambda_M(z)$ (Equation (\ref{eq:numerator})), hence converges at $\rho_\SL$ (separating two cases depending if $\rho_\SL<\rho_E$ or $\rho_\SL=\rho_E$).
\end{itemize}
It follows that the denominator $q$ must vanish: $q(\rho_\SL,\Lambda_M(\rho_\SL))=0$. Moreover, we have $\Lambda^f_M(z)<\Lambda_M(z)$ for all $z>0$ and all $M\in\mathcal P$. We can find $c<1$ such that $\Lambda^f_M(z)\le c\Lambda_M(z)$ for all $z\ge \frac12\rho_\SL$ and all $M\in\mathcal P$, and $\frac1c\notin \rho_\SL\cdot \mathrm{Spec}(A(\rho_\SL))$ (a finite set). Combining Lemma \ref{lem:monotone} and Equation (\ref{eq:denominator}), we have
\[ \Gamma_{\SL^f}(z)=r\bigl(z,(\Lambda^f_M(z))_M\bigr) \le r\bigl(z,(c\Lambda_M(z))_M\bigr) = \frac{p\bigl(z,(c\Lambda_M(z))_M\bigr)}{\det(I-czA(z))} \le \Gamma_\SL(z) \]
(for all $z$ where these series converge). Since the series $\Gamma_{\SL^f}$ is divergent, this ensures that no singularity is reached before $\rho_\SL$. Moreover, by definition of $c$, the denominator does not vanish at $z=\rho_\SL$. It follows that $\rho_{\SL^f}>\rho_\SL$, i.e., 
\[ \alpha(Q,X) \le\alpha(\SL^f)= \frac1{\rho_{\SL^f}}<\frac1{\rho_\SL}=\alpha(\FIM_X,X).\]

\section{Future work}
We believe that the results presented in this paper lay the groundwork for several future research project. In what follows, we discuss open problems and potential extensions that are either currently under active investigation or that we intend to work on the near-future.

\begin{enumerate}[leftmargin=6mm]
\item A primary objective for future work is to determine the precise language-theoretic complexity of the representative sets $\DGeo$, $\RGeo$, and $\LGeo$. Given that these languages are neither context-free nor co-context-free, it remains to be determined where they reside within the Chomsky hierarchy, specifically, whether they are context-sensitive or even if they belong to other classes of languages, such as indexed.

\item It is natural to try to extend the study to conjugacy by defining the conjugacy languages on free inverse monoids and develop an analogous framework as developed in this paper.

\item An important question to address is the dependency of these results on the chosen generating set. We aim to determine whether the complexity classifications identified for these languages are invariant under change of the generating set.
\end{enumerate}

\paragraph{AI statement.} The authors  declare that no   AI tools  were used to produce this manuscript and that this work was entirely produced by the authors.

\paragraph{Acknowledgments.} The authors would like to thank Tara Brough for crucial discussions that led to the proof of Theorem \ref{geo_CF}.

\paragraph{Funding.} The first author is supported by the Swiss SNF grant P500PT-225420. The second author was supported by national funds through the Fundação para a Ciência e Tecnologia, FCT, under the project
UID/04674/2025. The third author is supported by national funds through the FCT – Fundação para a Ciência e a Tecnologia, I.P., under the scope of the individual research grant 2025.03264.BD and the projects UID/297/2025 and UID/PRR/297/2025 (Center for Mathematics and Applications - NOVA Math).

\bibliographystyle{plain}
\bibliography{Bibliografia}

\end{document}